\documentclass[11pt,reqno,a4paper]{amsart}

\usepackage{amsmath,amssymb,amsthm,mathtools,mathrsfs,microtype,charter}
\usepackage{cite}
\usepackage[margin=.8in]{geometry}
\usepackage{enumitem,xcolor}
\usepackage[colorlinks=true,linkcolor=blue,citecolor=blue,urlcolor=blue]{hyperref}
\usepackage[british]{babel}
\usepackage[T1]{fontenc}
\usepackage[utf8]{inputenc}

\definecolor{RevisionBlue}{RGB}{0,0,0}
\DeclareRobustCommand{\rev}[1]{{\color{RevisionBlue}#1}}
\newenvironment{revision}{\begingroup\color{RevisionBlue}}{\endgroup}

\newtheorem{assumption}{Assumption}[section]
\newtheorem{theorem}[assumption]{Theorem}
\newtheorem{lemma}[assumption]{Lemma}
\newtheorem{proposition}[assumption]{Proposition}

\theoremstyle{definition}
\newtheorem{definition}[assumption]{Definition}
\newtheorem{remark}[assumption]{Remark}

\newcommand{\R}{\mathbb{R}}
\newcommand{\N}{\mathbb{N}}

\newcommand{\dt}{\partial_t}

\newcommand{\ds}{\,\mathrm{d}s}
\newcommand{\dx}{\,\mathrm{d}x}

\newcommand{\TM}{\mathcal T_M}
\newcommand{\ip}[2]{\left(#1,#2\right)}
\newcommand{\dual}[2]{\left\langle #1,#2\right\rangle}
\newcommand{\norm}[1]{\left\lVert #1\right\rVert}
\newcommand{\abs}[1]{\left\lvert #1\right\rvert}
\DeclareMathOperator{\diver}{div}

\begin{document}

\title[Rothe time discretization and weak solutions for a cutoff Westervelt system]{\Large Rothe time discretization and weak solutions \\[.1cm] for a cutoff Westervelt system}

\author{Marvin Fritz}

\address{%
Faculty of Mathematics, University of Vienna, 1090 Vienna, Austria
}

\maketitle
\vspace{-1cm}

\begin{abstract}
We study a fully implicit Rothe time discretization for a cutoff first-order formulation of the Westervelt equation. The key ingredients are the enthalpy variable and the primitive mobility variable, which \rev{turn each nonlinear time step into a uniformly monotone elliptic problem} and avoid higher-order energy estimates and inverse inequalities. For every time step, the discrete problem reduces to a monotone elliptic equation, which yields well-posedness of the Rothe scheme by the Browder--Minty theorem. We derive a discrete energy inequality, establish compactness for the transformed variable by an Alt--Luckhaus type argument, and pass to the limit to obtain existence of weak solutions for the cutoff first-order system. We prove a weak--strong uniqueness principle for the cutoff problem and formulate a conditional
a posteriori criterion under which the cutoff is inactive.  For sufficiently regular solutions of the forced cutoff problem, we establish consistency estimates which identify the temporal residuals produced by the Rothe approximation in the natural discrete-in-time weak norms. These estimates provide a rigorous consistency basis for the observed temporal behavior. Finally, numerical experiments illustrate the stability of the scheme, its observed near first-order temporal behavior, and \rev{inactivity of the cutoff along the computed discrete trajectories in the tested regimes}.%
\end{abstract} \vspace{.5cm}

The Westervelt equation is a classical model for nonlinear acoustic wave propagation and plays an
important role in applications such as high-intensity focused ultrasound, medical imaging and
therapy, lithotripsy, nondestructive testing, and underwater acoustics; see, for instance,
\cite{HamiltonBlackstock2024,SapozhnikovEtAl2019,Rudenko2022}. From the numerical point of view,
the equation is challenging because the coefficient in front of the highest time derivative depends
on the solution itself and may approach zero. This feature complicates both the well-posedness
analysis and the design of robust time discretizations. Existing approaches therefore often rely on
higher-order energy estimates, fixed-point arguments, and smallness assumptions that provide
$L^\infty$-control of the solution; see, for example,
\cite{KaltenbacherLasiecka2009,KaltenbacherNikolic2022,EggerFritz2025}. In parallel, a substantial
numerical literature has been developed for the Westervelt equation, including finite element,
mixed, discontinuous Galerkin, operator-splitting, and space--time methods; see
\cite{WalshTorres2007,KaltenbacherNikolicThalhammer2015,NikolicWohlmuth2019,AntoniettiEtAl2020,MelianiNikolic2024,GomezMeliani2025,GomezNikolic2025}.
Closely related nonlinear acoustic models from the same quasilinear family include, for instance,
the Kuznetsov \cite{kaltenbacher2011well,kaltenbacher2012analysis,dekkers2019cauchy}, Blackstock \cite{fritz2018well}, and Rasmussen-type equations \cite{EggerFritz2025}. 

The purpose of the present paper is to develop a time-discrete approach that is tailored to the
structure of the nonlinearity. Instead of working directly with the classical second-order equation,
we consider a first-order formulation in velocity and enthalpy variables and combine it with a
cutoff in the mobility coefficient. \begin{revision}
At each Rothe step, the cutoff formulation produces a uniformly monotone scalar elliptic problem and
naturally singles out the primitive mobility variable as the compact quantity in the limit process.
The scalar problem provides a robust nonlinear solve at each time level and yields a natural discrete
energy inequality.
\end{revision} In this way, the scheme can be analyzed in a low-order energy setting
without invoking the higher-order stability machinery that is standard in the classical quasilinear
wave framework.

The main contributions of this paper are as follows. First, we introduce a fully implicit Rothe
discretization for a cutoff first-order Westervelt system and prove well-posedness of each time
step by monotonicity methods. Second, we derive discrete energy estimates and compactness
properties for the Rothe approximations and pass to the limit to obtain existence of weak solutions
for the cutoff problem. Third, we prove a weak--strong uniqueness principle based on a relative
entropy for the transformed variable. Fourth, we formulate a conditional a posteriori criterion for
cutoff inactivity, which identifies a sufficient higher-order bound under which the uncut first-order
Westervelt system is recovered. 
The compactness argument is guided by the transformed variable $\nu(h)$ and is reminiscent of the
Alt--Luckhaus framework for nonlinear parabolic equations with time derivative of the form
$\partial_t b(u)$ \cite{AltLuckhaus1983}. Finally, we derive consistency estimates for sufficiently regular solutions of the forced cutoff
problem. These estimates show that the exact solution satisfies the Rothe equations up to residual
sequences of order $O(\tau)$ in the natural discrete $L^2$-in-time weak norms. They justify the interpretation of the observed rates as consistency-driven, but we do
not claim a full stability-based first-order error estimate in the energy norm. Numerical
experiments illustrate the stability of the scheme, its observed temporal behavior, and \rev{inactivity of the cutoff along the computed discrete trajectories}.

The paper is organized as follows. Section~\ref{sec:modeling} explains the passage from the classical
Westervelt equation to the cutoff first-order formulation and clarifies the reconstruction of the
scalar equation. Section~\ref{sec:cutoff} introduces the analytical formulation and the weak solution
concept. Section~\ref{sec:rothe} presents the Rothe scheme, derives
the discrete stability bounds, and then passes to the limit,
where the existence of weak solutions is proved first for $H_0^1$ initial data and then for general
$L^2$ initial data\rev{.} \rev{The general-$L^2$ result is obtained by approximation of the initial data at the level of continuous weak solutions.} Section~\ref{sec:weak-strong-uniqueness} proves a weak--strong uniqueness
principle for the cutoff system. Section~\ref{sec:cutoff-removal} discusses a posteriori removal of
the cutoff. Section~\ref{sec:error-estimate} establishes consistency estimates for the forced cutoff problem
and clarifies the stability obstruction that prevents a direct first-order energy-norm estimate
from the present argument. Section~\ref{sec:numerics} reports numerical experiments illustrating
the theory.

\section{From the Westervelt equation to the cutoff first-order formulation}\label{sec:modeling}

We briefly explain the passage from the classical Westervelt equation to the first-order cutoff
system studied in this paper.

Let $\psi$ denote the acoustic velocity potential. In a simplified nondimensional form, the
Westervelt equation can be written as
\begin{equation}\label{eq:westervelt-classical}
(1-a\psi_t)\psi_{tt} - b\Delta \psi_t - \Delta \psi = 0
\qquad\text{in }(0,T)\times\Omega,
\end{equation}
supplemented with homogeneous Dirichlet boundary conditions for $\psi_t$ and with suitable initial
data. The constants $a>0$ and $b>0$ describe the nonlinear and diffusive effects, respectively.

The coefficient $1-a\psi_t$ in front of the highest time derivative is the main structural
difficulty of the equation. In particular, the equation loses its hyperbolic--parabolic character
if this factor approaches zero. To rewrite the problem in first-order form, we introduce the
enthalpy variable
\[
h:=\psi_t
\]
and the velocity variable
\[
v:=-\nabla\psi.
\]
Differentiating the relation $v=-\nabla\psi$ with respect to time gives
\[
\partial_t v + \nabla h = 0.
\]
Moreover, substituting $h=\psi_t$ and $v=-\nabla\psi$ into \eqref{eq:westervelt-classical} yields
\[
(1-ah)\partial_t h + \diver v = b\Delta h.
\]
Hence the Westervelt equation is formally equivalent to the first-order system
\begin{equation}\label{eq:west-first-order}
\begin{aligned}
\partial_t v + \nabla h &= 0,
\\
(1-ah)\partial_t h + \diver v &= b\Delta h.
\end{aligned}
\end{equation}

For the analysis below, it is convenient to rewrite the second equation in conservative form.
To this end, we introduce the mobility
\[
m(r):=1-ar.
\]
Formally, \eqref{eq:west-first-order} then becomes
\[
\partial_t \nu(h) + \diver v = b\Delta h,
\qquad
\nu(r):=\int_0^r m(s)\,\mathrm ds
      = r-\frac{a}{2}r^2.
\]
However, the mobility $m(r)$ is not uniformly positive on all of $\R$, and therefore the map
$\nu$ is not globally monotone.

To avoid this degeneracy, we introduce a cutoff level $M\in(0,1/a)$ and define
\[
\TM(r):=\max\{-M,\min\{r,M\}\},
\qquad
m(r):=1-a\TM(r).
\]
Then there exist constants $\delta\in(0,1)$ and $\Lambda\ge 1$ such that
\[
\delta \le m(r)\le \Lambda
\qquad\text{for all }r\in\R.
\]
Explicitly, $\delta=1-aM$ and $\Lambda=1+aM$.
With the corresponding primitive
\begin{equation}\label{eq:def-nu-modeling}
\nu(r):=\int_0^r m(s)\,\mathrm ds,
\end{equation}
the cutoff problem takes the form
\begin{equation}\label{eq:cutoff-first-order-modeling}
\begin{aligned}
\partial_t v + \nabla h &= 0,
\\
\partial_t \nu(h) + \diver v &= b\Delta h.
\end{aligned}
\end{equation}

The advantage of \eqref{eq:cutoff-first-order-modeling} is that the transformed variable
$\nu(h)$ inherits the uniform monotonicity of the cutoff mobility. This yields a scalar monotone
elliptic problem at each Rothe step and provides the compact quantity for the passage to the limit.
Once a solution of the cutoff problem has been constructed, the uncut first-order Westervelt system
is recovered a posteriori if one can show that $|h|\le M$ almost everywhere, so that the cutoff is
inactive. Recovering the scalar Westervelt equation additionally requires the compatibility
$v_0=-\nabla\psi_0$ and the reconstruction of a velocity potential; see
Proposition~\ref{prop:potential-reconstruction}.

\begin{proposition}[Reconstruction of a velocity potential]\label{prop:potential-reconstruction}
Let $(v,h)$ satisfy
\[
v\in L^\infty(0,T;L^2(\Omega)^d)\cap H^1(0,T;L^2(\Omega)^d),
\]
\[
h\in L^\infty(0,T;L^2(\Omega))\cap L^2(0,T;H_0^1(\Omega)),
\qquad
\nu(h)\in H^1(0,T;H^{-1}(\Omega)),
\]
and suppose that $(v,h)$ solves the \emph{uncut} first-order Westervelt system
\begin{equation}\label{eq:uncut-first-order}
\begin{aligned}
\partial_t v + \nabla h &= 0
&&\text{in }L^2(0,T;L^2(\Omega)^d),\\
\partial_t \nu(h) + \diver v &= b\Delta h
&&\text{in }L^2(0,T;H^{-1}(\Omega)),
\end{aligned}
\end{equation}
where $\nu(r)=r-\frac a2 r^2$. Assume in addition that
\[
v(0)=-\nabla \psi_0
\qquad\text{for some }\psi_0\in H_0^1(\Omega).
\]
Define
\[
\psi(t):=\psi_0+\int_0^t h(s)\,\ds
\qquad\text{for }t\in[0,T].
\]
Then
\[
\psi\in W^{1,\infty}(0,T;L^2(\Omega))\cap W^{1,1}(0,T;H_0^1(\Omega)),
\qquad
\psi_t=h
\quad\text{a.e. in }(0,T)\times\Omega,
\]
and
\[
v=-\nabla\psi
\qquad\text{in }L^\infty(0,T;L^2(\Omega)^d).
\]
Consequently,
\[
\partial_t \nu(\psi_t)-\Delta\psi=b\Delta\psi_t
\qquad\text{in }L^2(0,T;H^{-1}(\Omega)).
\]
Since $\nu(r)=r-\frac a2 r^2$, this is the scalar Westervelt equation in conservative form.
\begin{revision}
If, in addition,
\[
\psi_t\in L^\infty((0,T)\times\Omega),
\qquad
\psi_{tt}\in L^2(0,T;L^2(\Omega)),
\]
then the Bochner chain rule gives
\[
\partial_t\!\left(\psi_t-\frac a2\psi_t^2\right)
=(1-a\psi_t)\psi_{tt}
\qquad\text{in }L^2(0,T;L^2(\Omega)),
\]
and the conservative identity reduces to
\[
(1-a\psi_t)\psi_{tt}-b\Delta\psi_t-\Delta\psi=0
\qquad\text{in }L^2(0,T;H^{-1}(\Omega)).
\]
\end{revision}
\end{proposition}

\begin{proof}
Because $h\in L^2(0,T;H_0^1(\Omega))\cap L^\infty(0,T;L^2(\Omega))$, the function
\[
\psi(t)=\psi_0+\int_0^t h(s)\,\ds
\]
belongs to
\(
W^{1,\infty}(0,T;L^2(\Omega))\cap W^{1,1}(0,T;H_0^1(\Omega))
\)
and thus it holds $\psi_t=h$ a.e.
Since it hold $v\in H^1(0,T;L^2(\Omega)^d)$, we use its continuous representative in
$C([0,T];L^2(\Omega)^d)$. The first equation in \eqref{eq:uncut-first-order} gives
\[
\partial_t v=-\nabla h
\qquad\text{in }L^2(0,T;L^2(\Omega)^d).
\]
Hence, for every $t\in[0,T]$,
\[
v(t)=v(0)-\int_0^t \nabla h(s)\,\ds
\qquad\text{in }L^2(\Omega)^d.
\]
Using the compatibility condition $v(0)=-\nabla\psi_0$ and the definition of $\psi$, we obtain
\[
v(t)
=
-\nabla\psi_0-\int_0^t \nabla h(s)\,\ds
=
-\nabla\psi(t)
\qquad\text{in }L^2(\Omega)^d
\]
for every $t\in[0,T]$. Thus $v=-\nabla\psi$ in
$L^\infty(0,T;L^2(\Omega)^d)$.
Substituting $h=\psi_t$ and $v=-\nabla\psi$ into the second equation of
\eqref{eq:uncut-first-order} yields
\[
\partial_t \nu(\psi_t)-\Delta\psi=b\Delta\psi_t
\qquad\text{in }L^2(0,T;H^{-1}(\Omega)).
\]
Using $\nu(r)=r-\frac a2 r^2$, we obtain the conservative scalar formulation
\[
\partial_t\!\left(\psi_t-\frac a2(\psi_t)^2\right)-\Delta\psi=b\Delta\psi_t
\qquad\text{in }L^2(0,T;H^{-1}(\Omega)).
\]
\begin{revision}
Under the additional assumptions
$\psi_t\in L^\infty((0,T)\times\Omega)$ and
$\psi_{tt}\in L^2(0,T;L^2(\Omega))$, the map
$r\mapsto r-\frac a2r^2$ has a bounded derivative on the essential range of $\psi_t$.
The Bochner chain rule therefore yields
\[
\partial_t\!\left(\psi_t-\frac a2(\psi_t)^2\right)
=(1-a\psi_t)\psi_{tt}
\qquad\text{in }L^2(0,T;L^2(\Omega)),
\]
which proves the final assertion.
\end{revision}
\end{proof}

\section{Analytical formulation of the cutoff problem}\label{sec:cutoff}

Let $\Omega\subset\R^d$, $d\le 3$, be a bounded Lipschitz domain, let $T>0$, and let $a,b>0$.
We define the cutoff
\[
\TM(r):=\max\bigl\{-M,\min\{r,M\}\bigr\}
\qquad (r\in\R),
\]
with some fixed $M\in(0,1/a)$, and set
\begin{equation}\label{eq:def-mu}
m(r):=1-a\TM(r).
\end{equation}
Then there exist constants $\delta\in(0,1)$ and $\Lambda\ge 1$ such that
\begin{equation}\label{eq:mobility-bounds}
\delta \le m(r) \le \Lambda
\qquad\text{for all } r\in\R.
\end{equation}
Explicitly, $\delta=1-aM$ and $\Lambda=1+aM$.
We introduce the primitives
\begin{equation}\label{eq:def-nu-E}
\nu(r):=\int_0^r m(s)\,\ds,
\qquad
E(r):=\int_0^r s\,m(s)\,\ds = r\nu(r)-\int_0^r \nu(s)\,\ds.
\end{equation}
By \eqref{eq:mobility-bounds}, the map $\nu:\R\to\R$ is strictly increasing and bi-Lipschitz. More precisely,
\begin{equation}\label{eq:nu-bounds}
\delta \abs{r-s} \le \abs{\nu(r)-\nu(s)} \le \Lambda \abs{r-s}
\qquad\text{for all } r,s\in\R,
\end{equation}
and
\begin{equation}\label{eq:E-bounds}
\frac{\delta}{2}\abs{r}^2 \le E(r) \le \frac{\Lambda}{2} \abs{r}^2
\qquad\text{for all } r\in\R.
\end{equation}

We consider the cutoff Westervelt system
\begin{equation}\label{eq:cutoff-westervelt}
\begin{aligned}
\dt v + \nabla h &= 0 &&\text{in } (0,T)\times\Omega,
\\
\dt \nu(h) + \diver v &= b\Delta h &&\text{in } (0,T)\times\Omega,
\\
h &= 0 &&\text{on } (0,T)\times\partial\Omega,
\\
v(0)=v_0, \qquad h(0)&=h_0 &&\text{in } \Omega.
\end{aligned}
\end{equation}
In weak form it is convenient to integrate the divergence term by parts and write the second equation as
\begin{equation}\label{eq:weak-second-form}
\dual{\dt\nu(h)}{\xi} - \ip{v}{\nabla\xi} + b\ip{\nabla h}{\nabla\xi}=0
\qquad\text{for all } \xi\in H_0^1(\Omega).
\end{equation}

\begin{definition}[Weak solution]\label{def:weak-solution}
A pair $(v,h)$ is called a weak solution of \eqref{eq:cutoff-westervelt} if
$$\begin{aligned}
v&\in L^\infty(0,T;L^2(\Omega)^d)\cap H^1(0,T;L^2(\Omega)^d),
\\
h&\in L^\infty(0,T;L^2(\Omega))\cap L^2(0,T;H_0^1(\Omega)),
\\
\nu(h)&\in L^2(0,T;H_0^1(\Omega))\cap H^1(0,T;H^{-1}(\Omega)),
\end{aligned}$$
and if for a.e. $t\in(0,T)$,
\begin{align}
\ip{\dt v(t)}{\varphi} + \ip{\nabla h(t)}{\varphi} &= 0
&&\forall\, \varphi\in L^2(\Omega)^d,
\label{eq:weak-v}
\\
\dual{\dt \nu(h)(t)}{\xi} - \ip{v(t)}{\nabla \xi} + b\ip{\nabla h(t)}{\nabla \xi} &= 0
&&\forall\, \xi\in H_0^1(\Omega),
\label{eq:weak-h}
\end{align}
and if the initial conditions are attained as
\[
v(0)=v_0 \quad\text{in } L^2(\Omega)^d,
\qquad
\nu(h)(0)=\nu(h_0) \quad\text{in } H^{-1}(\Omega).
\]
\begin{revision}
When $\nu(h)$ is interpreted through its $L^2$-continuous representative, both
$\nu(h)(0)$ and $\nu(h_0)$ belong to $L^2(\Omega)$. Since the canonical injection
$L^2(\Omega)\hookrightarrow H^{-1}(\Omega)$ is injective, equality in $H^{-1}(\Omega)$
implies equality of these two $L^2$ functions. Applying the globally Lipschitz Nemytskii map
$\nu^{-1}:L^2(\Omega)\to L^2(\Omega)$ then shows that the latter condition is equivalent to
$h(0)=h_0$ in $L^2(\Omega)$.
\end{revision}
\end{definition}

\begin{remark}[\rev{Why the enthalpy increment is essential}]
\begin{revision}
The formulation in terms of $\nu(h)$ is the key technical point. A coefficient-frozen increment
$m(h^{n-1})(h^n-h^{n-1})/\tau$ does \emph{not} inherit the energy inequality of
Lemma~\ref{lem:discrete-energy} in general. Indeed, testing such a step by $h^n$ would require the
pointwise inequality
\[
m(y)(x-y)x\ge E(x)-E(y),
\qquad x=h^n,\quad y=h^{n-1}.
\]
If $x,y\in[-M,M]$, then $m(s)=1-as$ and $E(s)=\frac12s^2-\frac a3s^3$, and direct calculation gives
\[
m(y)(x-y)x-\bigl(E(x)-E(y)\bigr)
=\frac{(x-y)^2}{6}\bigl(3+2a(x-y)\bigr),
\]
which can be negative. For example, $a=1$, $M=0.9$, $y=0.9$, and $x=-0.9$ give
$m(y)(x-y)x=0.162<0.486=E(x)-E(y)$. Thus uniform positivity of $m$ alone does not provide a
telescoping energy law for the frozen scheme. The comparison in Section~\ref{sec:numerics}
therefore treats its energy trace only as a numerical diagnostic; the proved energy inequality is
specific to the enthalpy increment $\nu(h^n)-\nu(h^{n-1})$.
\end{revision}
\end{remark}

\begin{remark}[Time continuity of the mobility variable]\label{rem:time-continuity-nu}
Since
\[
\nu(h)\in L^2(0,T;H_0^1(\Omega))\cap H^1(0,T;H^{-1}(\Omega)),
\]
the Lions--Magenes theorem \cite{LionsMagenes1972} gives
\(
\nu(h)\in C([0,T];L^2(\Omega)).
\)
Because $\nu^{-1}$ is globally Lipschitz, this also yields
\[
h=\nu^{-1}(\nu(h))\in C([0,T];L^2(\Omega)).
\]
Thus the initial condition may equivalently be written as $h(0)=h_0$ in $L^2(\Omega)$.
\end{remark}

\begin{remark}[Relation to the Alt--Luckhaus framework]
The formulation in terms of the primitive $\nu(h)$ is reminiscent of the classical paper \cite{AltLuckhaus1983} of Alt and
Luckhaus on quasilinear elliptic--parabolic equations with time derivative of the form
$\partial_t b(u)$. In our setting, the map $\nu$ plays the role of the monotone function $b$, and
the implicit time step for $h^n$,
\[
\frac{\nu(h^n)-\nu(h^{n-1})}{\tau} -\diver\big((b+\tau)\nabla h^n\big)
  =
  -\diver v^{n-1},
\]
is a monotone elliptic problem of the same general type.
There is, however, an important structural difference. The full cutoff Westervelt system is not a
single quasilinear parabolic equation for $h$, but a coupled first-order system for $(v,h)$.
Eliminating $v$ leads formally to
\[
\dt \nu(h) - b\Delta h
=
-\diver v_0 + \int_0^t \Delta h(s)\,\ds,
\]
that is, to a parabolic equation with memory. Hence the continuous problem does not fall directly
into the local elliptic--parabolic framework of Alt--Luckhaus. Nevertheless, the Rothe
construction and the compactness strategy are closely aligned with that theory: one first obtains
compactness for $\nu(h_\tau)$ and then transfers it to $h_\tau$ through the bi-Lipschitz continuity
of $\nu^{-1}$.
\end{remark}

\section{Rothe scheme}\label{sec:rothe}

Fix $N\in\N$ and let $\tau=T/N$, $t^n=n\tau$. Given $v^0=v_0\in L^2(\Omega)^d$ and
$h^0=h_0\in L^2(\Omega)$, we construct $(v^n,h^n)_{n=1}^N$ recursively as follows:

Given $(v^{n-1},h^{n-1})$, find $h^n\in H_0^1(\Omega)$ such that
\begin{equation}\label{eq:rothe-scalar}
\frac{1}{\tau}\ip{\nu(h^n)-\nu(h^{n-1})}{\xi}
+(b+\tau)\ip{\nabla h^n}{\nabla\xi}
= \ip{v^{n-1}}{\nabla\xi}
\qquad\forall\,\xi\in H_0^1(\Omega),
\end{equation}
and then define
\begin{equation}\label{eq:rothe-velocity}
v^n := v^{n-1} - \tau \nabla h^n.
\end{equation}
Equivalently, $(v^n,h^n)$ satisfies
\begin{equation}\label{eq:rothe-system}
\begin{aligned}
\frac{1}{\tau} \ip{v^n-v^{n-1}}{\varphi} + \ip{\nabla h^n}{\varphi} &=0
&&\forall\,\varphi\in L^2(\Omega)^d,
\\
\frac{1}{\tau}\ip{\nu(h^n)-\nu(h^{n-1})}{\xi} - \ip{v^n}{\nabla\xi} + b\ip{\nabla h^n}{\nabla\xi} &=0
&&\forall\,\xi\in H_0^1(\Omega).
\end{aligned}
\end{equation}

\begin{proposition}[Existence of one time step]
For every $n\in\{1,\dots,N\}$ and every given $(v^{n-1},h^{n-1})\in L^2(\Omega)^d\times L^2(\Omega)$,
the variational problem \eqref{eq:rothe-scalar} admits a unique solution $h^n\in H_0^1(\Omega)$.
Consequently, $v^n$ is uniquely defined by \eqref{eq:rothe-velocity}.
\end{proposition}

\begin{proof}
Define the operator $A:H_0^1(\Omega)\to H^{-1}(\Omega)$ by
\[
\dual{A(w)}{\xi}
:= \frac1\tau\ip{\nu(w)}{\xi} + (b+\tau)\ip{\nabla w}{\nabla\xi}.
\]
Since $\nu$ is continuous and monotone, the operator $A$ is hemicontinuous and monotone. Since
$\nu$ is globally Lipschitz and $\nu(0)=0$, the Nemytskii map $w\mapsto\nu(w)$ maps bounded
subsets of $L^2(\Omega)$ into bounded subsets of $L^2(\Omega)$. Consequently $A$ maps bounded
subsets of $H_0^1(\Omega)$ into bounded subsets of $H^{-1}(\Omega)$. Moreover,
by \eqref{eq:nu-bounds},
\[
\dual{A(w)-A(z)}{w-z}
= \frac1\tau\ip{\nu(w)-\nu(z)}{w-z} + (b+\tau)\norm{\nabla(w-z)}_{L^2(\Omega)}^2
\ge \frac{\delta}{\tau}\norm{w-z}_{L^2(\Omega)}^2,
\]
so $A$ is strictly monotone. Finally, using \eqref{eq:nu-bounds} and Poincaré's inequality,
\[
\dual{A(w)}{w}
= \frac1\tau \ip{\nu(w)}{w} + (b+\tau)\norm{\nabla w}_{L^2(\Omega)}^2
\ge \frac{\delta}{\tau}\norm{w}_{L^2(\Omega)}^2 + (b+\tau)\norm{\nabla w}_{L^2(\Omega)}^2,
\]
which yields coercivity on $H_0^1(\Omega)$. The right-hand side of \eqref{eq:rothe-scalar},
\[
\xi\mapsto \frac1\tau\ip{\nu(h^{n-1})}{\xi} + \ip{v^{n-1}}{\nabla\xi},
\]
is a bounded linear functional on $H_0^1(\Omega)$. By the Browder--Minty theorem \cite{ZeidlerIIB}, there exists a
unique $h^n\in H_0^1(\Omega)$ solving \eqref{eq:rothe-scalar}.
\end{proof}

\subsection{Discrete a priori estimates}

The next lemma provides the basic energy inequality.

\begin{lemma}[Discrete energy inequality]\label{lem:discrete-energy}
The Rothe iterates satisfy, for every $m\in\{1,\dots,N\}$,
\begin{align}
\!\int_\Omega\! E(h^m)\dx + \frac12 \norm{v^m}_{L^2(\Omega)}^2
+ b\tau \sum_{n=1}^m \norm{\nabla h^n}_{L^2(\Omega)}^2
+ \frac12 \sum_{n=1}^m \norm{v^n\!-\!v^{n-1}}_{L^2(\Omega)}^2
\le
\!\int_\Omega\! E(h^0)\dx + \frac12 \norm{v^0}_{L^2(\Omega)}^2
\label{eq:disc-energy}
\end{align}
In particular,
\begin{equation}\label{eq:uniform-bounds-discrete}
\sup_{1\le n\le N} \norm{v^n}_{L^2(\Omega)}^2
+ \sup_{1\le n\le N} \norm{h^n}_{L^2(\Omega)}^2
+ \tau\sum_{n=1}^N \norm{\nabla h^n}_{L^2(\Omega)}^2
\le C,
\end{equation}
with a constant $C$ independent of $\tau$.
\end{lemma}

\begin{proof}
We test \eqref{eq:rothe-scalar} with $\xi=h^n$. This gives
\begin{equation}\label{eq:test-hn}
\ip{\nu(h^n)-\nu(h^{n-1})}{h^n}
+ \tau(b+\tau)\norm{\nabla h^n}_{L^2(\Omega)}^2
= \tau\ip{v^{n-1}}{\nabla h^n}.
\end{equation}
Define
\[
\Psi(y):=\int_0^y \nu^{-1}(s)\,\ds.
\]
Since $\nu^{-1}$ is increasing, $\Psi$ is convex. Moreover,
\[
\Psi(\nu(r))=\int_0^{\nu(r)} \nu^{-1}(s)\,\ds = \int_0^r s\,m(s)\,\ds = E(r).
\]
Therefore, by convexity of $\Psi$,
\begin{equation}\label{eq:convex-E-step}
\ip{\nu(h^n)-\nu(h^{n-1})}{h^n}
\ge \int_\Omega E(h^n)\dx - \int_\Omega E(h^{n-1})\dx.
\end{equation}
On the other hand, the update formula \eqref{eq:rothe-velocity} implies
\begin{equation}\label{eq:vel-jump-id}
v^n-v^{n-1} = -\tau\nabla h^n,
\qquad
\norm{v^n-v^{n-1}}_{L^2(\Omega)}^2 = \tau^2\norm{\nabla h^n}_{L^2(\Omega)}^2,
\end{equation}
and therefore, by the polarization identity,
\begin{equation}\label{eq:vel-polarization}
\tau\ip{v^{n-1}}{\nabla h^n}
= -\ip{v^{n-1}}{v^n-v^{n-1}}
= \frac12\norm{v^{n-1}}_{L^2(\Omega)}^2 - \frac12\norm{v^n}_{L^2(\Omega)}^2
+ \frac12\norm{v^n-v^{n-1}}_{L^2(\Omega)}^2.
\end{equation}
Inserting \eqref{eq:vel-polarization} into \eqref{eq:test-hn} and using
\eqref{eq:convex-E-step}, and splitting the diffusion coefficient as
$\tau(b+\tau)\norm{\nabla h^n}^2 = b\tau\norm{\nabla h^n}^2+\tau^2\norm{\nabla h^n}^2$,
yields
\begin{align*}
&\int_\Omega E(h^n)\dx - \int_\Omega E(h^{n-1})\dx
+ \frac12\norm{v^n}_{L^2(\Omega)}^2 - \frac12\norm{v^{n-1}}_{L^2(\Omega)}^2
+ b\tau\norm{\nabla h^n}_{L^2(\Omega)}^2
\\ &\quad+ \tau^2\norm{\nabla h^n}_{L^2(\Omega)}^2
- \frac12\norm{v^n-v^{n-1}}_{L^2(\Omega)}^2
\le 0.
\end{align*}
Here the term $\tau^2\norm{\nabla h^n}^2$ originates from the implicit stabilization $b+\tau$ in
\eqref{eq:rothe-scalar}, whereas the term $-\tfrac12\norm{v^n-v^{n-1}}^2$ comes from the
polarization identity \eqref{eq:vel-polarization}; the two are combined using the
\emph{independent} velocity-jump identity \eqref{eq:vel-jump-id}.
By the second identity in \eqref{eq:vel-jump-id}, the last two terms on the
left-hand side combine into
\[
\tau^2\norm{\nabla h^n}_{L^2(\Omega)}^2
- \frac12\norm{v^n-v^{n-1}}_{L^2(\Omega)}^2
=
\frac12\norm{v^n-v^{n-1}}_{L^2(\Omega)}^2,
\]
so that
\begin{align*}
&\int_\Omega E(h^n)\dx - \int_\Omega E(h^{n-1})\dx
+ \frac12\norm{v^n}_{L^2(\Omega)}^2 - \frac12\norm{v^{n-1}}_{L^2(\Omega)}^2
\\ &\quad+ b\tau\norm{\nabla h^n}_{L^2(\Omega)}^2
+ \frac12\norm{v^n-v^{n-1}}_{L^2(\Omega)}^2
\le 0.
\end{align*}
Summation from $n=1$ to $n=m$ yields \eqref{eq:disc-energy}. Finally,
\eqref{eq:uniform-bounds-discrete} follows from \eqref{eq:E-bounds}.
\end{proof}

We next estimate the discrete time derivative of $\nu(h)$. But first we clarify some notation.

Given discrete values $(v^n,h^n)_{n=0}^N$, we introduce the standard Rothe interpolants.
For $t\in(t^{n-1},t^n]$, we define the piecewise constant functions
\[
\bar h_\tau(t):=h^n,
\qquad
\bar v_\tau(t):=v^n,
\qquad
\bar v_\tau^-(t):=v^{n-1}.
\]
We further define the piecewise affine interpolants
\[
\widehat v_\tau(t)
:=
\frac{t-t^{n-1}}{\tau}v^n+\frac{t^n-t}{\tau}v^{n-1},
\qquad
\widehat \nu_\tau(t)
:=
\frac{t-t^{n-1}}{\tau}\nu(h^n)+\frac{t^n-t}{\tau}\nu(h^{n-1}),
\]
for $t\in[t^{n-1},t^n]$.
Then, by construction,
\[
\partial_t \widehat v_\tau(t)=\frac{v^n-v^{n-1}}{\tau},
\qquad
\partial_t \widehat \nu_\tau(t)=\frac{\nu(h^n)-\nu(h^{n-1})}{\tau}
\qquad\text{for a.e. } t\in(t^{n-1},t^n).
\]

\begin{lemma}[Discrete Aubin--Lions compactness for Rothe interpolants]\label{lem:discrete-aubin-lions}
Let
\(
X_0 \hookrightarrow\hookrightarrow X \hookrightarrow X_1
\)
be Banach spaces with $X_0$ reflexive, with compact embedding $X_0\hookrightarrow X$ and continuous embedding
$X\hookrightarrow X_1$. For each $\tau=T/N$, let $(w^n)_{n=0}^N\subset X_0$ and denote by
$\bar w_\tau$ and $\widehat w_\tau$ the associated piecewise constant and piecewise affine
interpolants. Assume that
\begin{equation}\label{eq:dal-ass}
\|\bar w_\tau\|_{L^2(0,T;X_0)}
+
\|\partial_t \widehat w_\tau\|_{L^2(0,T;X_1)}
+
\|w^0\|_{X_0}
\le C
\end{equation}
uniformly in $\tau$. Then $(\bar w_\tau)_\tau$ is relatively compact in $L^2(0,T;X)$.
Moreover,
\begin{equation}\label{eq:dal-difference}
\|\widehat w_\tau-\bar w_\tau\|_{L^2(0,T;X_1)}
\le
\tau \|\partial_t \widehat w_\tau\|_{L^2(0,T;X_1)}.
\end{equation}
In particular, after extraction of a subsequence,
\[
\bar w_\tau \to w
\qquad\text{strongly in }L^2(0,T;X)
\]
for some $w\in L^2(0,T;X_0)$.
\end{lemma}

\begin{remark}[Relation to standard Rothe compactness arguments]
Lemma~\ref{lem:discrete-aubin-lions} is a Rothe-type compactness statement for piecewise constant
interpolants. It is closely related to the standard compactness toolbox for time-discrete
approximations and to the classical comparison between piecewise affine and piecewise constant
interpolants used in Rothe methods; see, for example,
\cite[Section~11.1]{schweizer2013partielle}.
\end{remark}

\begin{proof}
For $t\in(t^{n-1},t^n]$ one has
\[
\widehat w_\tau(t)-\bar w_\tau(t)
=
\frac{t^n-t}{\tau}(w^{n-1}-w^n)
=
-(t^n-t)\,\partial_t\widehat w_\tau(t).
\]
Hence \eqref{eq:dal-difference} follows immediately.
Next we show that $(\widehat w_\tau)_\tau$ is bounded in $L^2(0,T;X_0)$. On each interval
$[t^{n-1},t^n]$ we may write
\[
\widehat w_\tau(t)=\theta w^n+(1-\theta)w^{n-1},
\qquad
\theta=\frac{t-t^{n-1}}{\tau}\in[0,1].
\]
Since the square of the norm is convex,
\[
\|\widehat w_\tau(t)\|_{X_0}^2
\le
\theta \|w^n\|_{X_0}^2+(1-\theta)\|w^{n-1}\|_{X_0}^2.
\]
Integrating over $[t^{n-1},t^n]$ and summing over $n$ yields
\[
\|\widehat w_\tau\|_{L^2(0,T;X_0)}^2
\le
\|\bar w_\tau\|_{L^2(0,T;X_0)}^2 + \frac{\tau}{2}\|w^0\|_{X_0}^2
\le C.
\]
Together with \eqref{eq:dal-ass}, this shows that $(\widehat w_\tau)_\tau$ is bounded in
\[
L^2(0,T;X_0)\cap H^1(0,T;X_1).
\]
\begin{revision}
Since $X_0$ is reflexive, after extraction of a subsequence there is
$w_0\in L^2(0,T;X_0)$ such that
\[
\widehat w_\tau\rightharpoonup w_0
\qquad\text{weakly in }L^2(0,T;X_0).
\]
The Aubin--Lions--Simon compactness theorem \cite{Simon1987}, applied to this subsequence, yields a
further subsequence and a function $w\in L^2(0,T;X)$ such that
\[
\widehat w_\tau\to w
\qquad\text{strongly in }L^2(0,T;X).
\]
Because the embedding $X_0\hookrightarrow X$ is continuous, the weak $L^2(0,T;X_0)$ limit has image
$w_0$ as a weak $L^2(0,T;X)$ limit. Uniqueness of the weak limit in $L^2(0,T;X)$ gives
$w=w_0$; in particular, $w\in L^2(0,T;X_0)$.
\end{revision}
To obtain compactness of the piecewise constant interpolants, let $0<\sigma<T$. For a.e.\
$t\in(0,T-\sigma)$,
\begin{align*}
\bar w_\tau(t+\sigma)-\bar w_\tau(t)
&=
\widehat w_\tau(t+\sigma)-\widehat w_\tau(t)
+
\bigl(\bar w_\tau(t+\sigma)-\widehat w_\tau(t+\sigma)\bigr)
-
\bigl(\bar w_\tau(t)-\widehat w_\tau(t)\bigr)
\\
&=
\int_t^{t+\sigma}\partial_s\widehat w_\tau(s)\,\ds
+
\bigl(\bar w_\tau(t+\sigma)-\widehat w_\tau(t+\sigma)\bigr)
-
\bigl(\bar w_\tau(t)-\widehat w_\tau(t)\bigr).
\end{align*}
Therefore, by Minkowski's inequality and \eqref{eq:dal-difference},
\begin{revision}
\begin{align*}
\|\bar w_\tau(\cdot+\sigma)-\bar w_\tau(\cdot)\|_{L^2(0,T-\sigma;X_1)}
&\le
\left\|
\int_\cdot^{\cdot+\sigma}\partial_s\widehat w_\tau(s)\,\ds
\right\|_{L^2(0,T-\sigma;X_1)}
+
2\|\bar w_\tau-\widehat w_\tau\|_{L^2(0,T;X_1)}
\\
&\le
\sigma\|\partial_t\widehat w_\tau\|_{L^2(0,T;X_1)}
+
2\tau\|\partial_t\widehat w_\tau\|_{L^2(0,T;X_1)}
\\
&\le
C\bigl(\sigma+\tau\bigr).
\end{align*}
\end{revision}
\begin{revision}
The family $(\bar w_\tau)_\tau$ is bounded in $L^2(0,T;X_0)$ by assumption. To see that its time
translates vanish uniformly, fix $\varepsilon>0$. Among the time steps $\tau=T/N$, only finitely
many satisfy $\tau\ge\varepsilon$. For each of those finitely many fixed functions, translations
vanish in $L^2(0,T;X_1)$ as $\sigma\to0$. For all remaining time steps, the preceding estimate gives
\[
\|\bar w_\tau(\cdot+\sigma)-\bar w_\tau(\cdot)\|_{L^2(0,T-\sigma;X_1)}
\le C(\sigma+\varepsilon).
\]
Consequently,
\[
\limsup_{\sigma\to0}\sup_\tau
\|\bar w_\tau(\cdot+\sigma)-\bar w_\tau(\cdot)\|_{L^2(0,T-\sigma;X_1)}
\le C\varepsilon.
\]
Letting $\varepsilon\downarrow0$ proves uniform decay of the time translates. Simon's time-translate
compactness criterion \cite{Simon1987} therefore shows that $(\bar w_\tau)_\tau$ is relatively compact
in $L^2(0,T;X)$.

We now extract this compactness subsequence from the subsequence already fixed for
$\widehat w_\tau$, and write
\[
\bar w_\tau\to \bar w
\qquad\text{strongly in }L^2(0,T;X).
\]
By \eqref{eq:dal-difference}, $\widehat w_\tau-\bar w_\tau\to0$ in $L^2(0,T;X_1)$. Since
$\widehat w_\tau\to w$ in $L^2(0,T;X)$ and $X\hookrightarrow X_1$ continuously, both interpolants
have the same distributional limit. Hence $\bar w=w$, which proves the final assertion.
\end{revision}
\end{proof}

\begin{lemma}[Bound for the discrete time derivative]\label{lem:time-derivative}
Let
\[
\bar h_\tau(t):=h^n,
\qquad
\bar v_\tau(t):=v^n,
\qquad
\bar v_\tau^-(t):=v^{n-1}
\qquad \text{for } t\in(t^{n-1},t^n],
\]
and let $\widehat{\nu}_\tau$ be the piecewise affine interpolant of the values $\nu(h^n)$. Then
\begin{equation}\label{eq:dtnu-bound}
\norm{\dt \widehat{\nu}_\tau}_{L^2(0,T;H^{-1}(\Omega))}
\le C,
\end{equation}
with a constant $C$ independent of $\tau$.
\end{lemma}

\begin{proof}
For $t\in(t^{n-1},t^n]$, we recall that the derivative $\dt\widehat{\nu}_\tau$ is given by
\(
\dt\widehat{\nu}_\tau(t)=\frac{1}{\tau} (\nu(h^n)-\nu(h^{n-1})).
\)
Using \eqref{eq:rothe-scalar}, for every $\xi\in H_0^1(\Omega)$ we obtain
\[
\dual{\dt\widehat{\nu}_\tau(t)}{\xi}
= \ip{\bar v_\tau^-(t)}{\nabla\xi} - (b+\tau)\ip{\nabla\bar h_\tau(t)}{\nabla\xi}.
\]
Hence
\[
\norm{\dt\widehat{\nu}_\tau(t)}_{H^{-1}(\Omega)}
\le \norm{\bar v_\tau^-(t)}_{L^2(\Omega)} + (b+\tau)\norm{\nabla \bar h_\tau(t)}_{L^2(\Omega)}.
\]
Squaring and integrating over $(0,T)$, and using \eqref{eq:uniform-bounds-discrete}, proves
\eqref{eq:dtnu-bound}.
\end{proof}

For the velocity variable we also record the immediate consequence of \eqref{eq:rothe-velocity}.

\begin{lemma}[Time regularity of the velocity]\label{lem:v-time-reg}
Let $\widehat v_\tau$ be the piecewise affine interpolant of the values $v^n$. Then
\begin{equation}\label{eq:v-time-bound}
\dt \widehat v_\tau = -\nabla \bar h_\tau
\qquad\text{a.e. in } (0,T),
\end{equation}
and therefore
\begin{equation}\label{eq:v-H1-bound}
\norm{\widehat v_\tau}_{H^1(0,T;L^2(\Omega)^d)} \le C.
\end{equation}
\end{lemma}

\begin{revision}
\begin{proof}
The identity \eqref{eq:v-time-bound} follows directly from
$v^n-v^{n-1}=-\tau\nabla h^n$ on every interval $(t^{n-1},t^n)$. Hence
\[
\|\partial_t\widehat v_\tau\|_{L^2(0,T;L^2(\Omega)^d)}
=\|\nabla\bar h_\tau\|_{L^2(0,T;L^2(\Omega)^d)}\le C
\]
by Lemma~\ref{lem:discrete-energy}. Moreover, convexity of the squared norm gives, on
$[t^{n-1},t^n]$,
\[
\|\widehat v_\tau(t)\|_{L^2(\Omega)^d}^2
\le \theta\|v^n\|_{L^2(\Omega)^d}^2
 +(1-\theta)\|v^{n-1}\|_{L^2(\Omega)^d}^2,
\qquad \theta=\frac{t-t^{n-1}}{\tau}.
\]
After integration and summation, the nodal energy bound yields
$\|\widehat v_\tau\|_{L^2(0,T;L^2(\Omega)^d)}\le C$. Combining the two estimates proves
\eqref{eq:v-H1-bound}.
\end{proof}
\end{revision}

\begin{lemma}[Bounds for the mobility variable]\label{lem:nu-spatial-bounds}
Assume $h_0\in H_0^1(\Omega)$. Then, for every $n=0,\dots,N$, one has $\nu(h^n)\in H_0^1(\Omega)$ and
\begin{equation}\label{eq:nuhn-bound}
\sup_{0\le n\le N}\norm{\nu(h^n)}_{L^2(\Omega)}^2
+
\tau\sum_{n=0}^N \norm{\nabla \nu(h^n)}_{L^2(\Omega)}^2
\le C,
\end{equation}
\rev{with a constant $C$ independent of $\tau$ but depending, in addition to the energy data and the fixed parameters, on $T$ and $\|h_0\|_{H_0^1(\Omega)}$.}
Moreover,
\begin{equation}\label{eq:nu-interpolants-bound}
\nu(\bar h_\tau)\ \text{and}\ \widehat\nu_\tau
\quad\text{are bounded in }L^2(0,T;H_0^1(\Omega)).
\end{equation}
\end{lemma}

\begin{proof}
Since $\nu$ is globally Lipschitz and $\nu(0)=0$, the Sobolev chain rule for Lipschitz functions
\cite{P.Ziemer1989} implies that $\nu(h^n)\in H_0^1(\Omega)$ and
\[
\nabla \nu(h^n)=\nu'(h^n)\nabla h^n = m(h^n)\nabla h^n
\qquad\text{a.e. in }\Omega.
\]
By \eqref{eq:mobility-bounds} and \eqref{eq:uniform-bounds-discrete},
\[
\norm{\nu(h^n)}_{L^2(\Omega)} \le \Lambda\norm{h^n}_{L^2(\Omega)},
\qquad
\norm{\nabla \nu(h^n)}_{L^2(\Omega)} \le \Lambda\norm{\nabla h^n}_{L^2(\Omega)}\rev{.}
\]
\begin{revision}
The energy estimate controls the gradient terms for $n\ge1$, while the initial term is estimated
separately. Thus
\[
\tau\sum_{n=0}^N\|\nabla\nu(h^n)\|_{L^2(\Omega)}^2
\le
\tau\Lambda^2\|\nabla h_0\|_{L^2(\Omega)}^2
+
\Lambda^2\tau\sum_{n=1}^N\|\nabla h^n\|_{L^2(\Omega)}^2
\le C.
\]
The $L^2$ part of the time-integrated $H_0^1$ norm follows from
\[
\tau\sum_{n=0}^N\|\nu(h^n)\|_{L^2(\Omega)}^2
\le (T+\tau)\sup_{0\le n\le N}\|\nu(h^n)\|_{L^2(\Omega)}^2.
\]
Together with the displayed pointwise bounds, these estimates prove \eqref{eq:nuhn-bound} and the
bound for $\nu(\bar h_\tau)$. For $\widehat\nu_\tau$, note that on each interval
$[t^{n-1},t^n]$,
\end{revision}
\[
\widehat\nu_\tau(t)
=
\theta \nu(h^n)+(1-\theta)\nu(h^{n-1}),
\qquad
\theta=\frac{t-t^{n-1}}{\tau}\in[0,1].
\]
Since the square of the norm is convex, we obtain
\[
\norm{\widehat\nu_\tau(t)}_{H_0^1(\Omega)}^2
\le
\theta \norm{\nu(h^n)}_{H_0^1(\Omega)}^2
+
(1-\theta)\norm{\nu(h^{n-1})}_{H_0^1(\Omega)}^2.
\]
Integrating over $[t^{n-1},t^n]$ and summing in $n$ yields
\[
\norm{\widehat\nu_\tau}_{L^2(0,T;H_0^1(\Omega))}^2
\le
\tau\sum_{n=0}^N \norm{\nu(h^n)}_{H_0^1(\Omega)}^2
\le C,
\]
which proves \eqref{eq:nu-interpolants-bound}.
\end{proof}

\begin{lemma}[Compactness of the Rothe interpolants]\label{lem:compactness}
Let $(v^n,h^n)_{n=0}^N$ be Rothe iterates associated with data
\[
v_0\in L^2(\Omega)^d,
\qquad
h_0\in H_0^1(\Omega),
\]
and let $\bar v_\tau,\widehat v_\tau,\bar h_\tau,\widehat\nu_\tau$ be the associated interpolants.
Assume the uniform bounds of Lemmas~\ref{lem:discrete-energy}, \ref{lem:time-derivative},
\ref{lem:v-time-reg}, and \ref{lem:nu-spatial-bounds}. Then there exist a subsequence, not
relabeled, and functions
\[
v\in H^1(0,T;L^2(\Omega)^d)\cap L^\infty(0,T;L^2(\Omega)^d),
\qquad
h\in L^\infty(0,T;L^2(\Omega))\cap L^2(0,T;H_0^1(\Omega)),
\]
such that
\begin{align}
\widehat v_\tau &\rightharpoonup v
&&\text{weakly in }H^1(0,T;L^2(\Omega)^d), \label{eq:cv-vhat}\\
\widehat v_\tau &\to v
&&\text{strongly in }C([0,T];H^{-1}(\Omega)^d), \label{eq:cv-vhat-strong}\\
\bar v_\tau &\stackrel{*}{\rightharpoonup} v
&&\text{weakly-* in }L^\infty(0,T;L^2(\Omega)^d), \label{eq:cv-vbar-star}\\
\bar v_\tau &\rightharpoonup v
&&\text{weakly in }L^2(0,T;L^2(\Omega)^d), \label{eq:cv-vbar}\\
\widehat\nu_\tau &\rightharpoonup \nu(h)
&&\text{weakly in }L^2(0,T;H_0^1(\Omega)), \label{eq:cv-nuhat-weak}\\
\widehat\nu_\tau &\rightharpoonup \nu(h)
&&\text{weakly in }H^1(0,T;H^{-1}(\Omega)), \label{eq:cv-nuhat-time}\\
\widehat\nu_\tau &\to \nu(h)
&&\text{strongly in }L^2(0,T;L^2(\Omega)), \label{eq:cv-nuhat-strong}\\
\nu(\bar h_\tau) &\to \nu(h)
&&\text{strongly in }L^2(0,T;L^2(\Omega)), \label{eq:cv-nubar-strong}\\
\bar h_\tau &\stackrel{*}{\rightharpoonup} h
&&\text{weakly-* in }L^\infty(0,T;L^2(\Omega)), \label{eq:cv-hbar-star}\\
\bar h_\tau &\rightharpoonup h
&&\text{weakly in }L^2(0,T;H_0^1(\Omega)), \label{eq:cv-hbar-weak}\\
\bar h_\tau &\to h
&&\text{strongly in }L^2(0,T;L^2(\Omega)). \label{eq:cv-hbar-strong}
\end{align}
\end{lemma}

\begin{proof}
By Lemmas~\ref{lem:discrete-energy}, \ref{lem:time-derivative}, \ref{lem:v-time-reg}, and
\ref{lem:nu-spatial-bounds}, the families
\[
\begin{aligned}
\widehat v_\tau &\subset H^1(0,T;L^2(\Omega)^d)\cap L^\infty(0,T;L^2(\Omega)^d),
\qquad
&&\bar v_\tau \subset L^\infty(0,T;L^2(\Omega)^d),
\\
\widehat\nu_\tau &\subset H^1(0,T;H^{-1}(\Omega))\cap L^2(0,T;H_0^1(\Omega)),
\qquad
&&\bar h_\tau \subset L^\infty(0,T;L^2(\Omega))\cap L^2(0,T;H_0^1(\Omega)),
\end{aligned}
\]
are bounded. Hence, by Banach--Alaoglu and reflexivity, after extraction of a subsequence there
exist functions
\[
v\in H^1(0,T;L^2(\Omega)^d)\cap L^\infty(0,T;L^2(\Omega)^d),
\qquad
w\in H^1(0,T;H^{-1}(\Omega))\cap L^2(0,T;H_0^1(\Omega)),
\]
and
\[
h\in L^\infty(0,T;L^2(\Omega))\cap L^2(0,T;H_0^1(\Omega))
\]
\begin{revision}
such that \eqref{eq:cv-vhat}, \eqref{eq:cv-vbar-star}, \eqref{eq:cv-hbar-star}, and
\eqref{eq:cv-hbar-weak} hold, while
\[
\widehat\nu_\tau\rightharpoonup w
\quad\text{weakly in }L^2(0,T;H_0^1(\Omega))
\quad\text{and in }H^1(0,T;H^{-1}(\Omega)).
\]
At this stage the identity $w=\nu(h)$ has not yet been used; it is established below.
\end{revision}
Since
\[
\widehat\nu_\tau \subset H^1(0,T;H^{-1}(\Omega))\cap L^2(0,T;H_0^1(\Omega))
\]
is bounded, the Aubin--Lions--Simon compactness theorem \cite{Simon1987} yields, after passing to a subsequence,
\[
\widehat\nu_\tau \to w
\qquad\text{strongly in }L^2(0,T;L^2(\Omega)).
\]
On the other hand, Lemma~\ref{lem:discrete-aubin-lions} applied to the nodal values
$w^n:=\nu(h^n)$ with
\[
X_0=H_0^1(\Omega),\qquad X=L^2(\Omega),\qquad X_1=H^{-1}(\Omega)
\]
and with $w^0=\nu(h^0)=\nu(h_0)\in H_0^1(\Omega)$, which is fixed independently of $\tau$,
shows that, after extraction of a further subsequence,
\[
\nu(\bar h_\tau)\to \widetilde w
\qquad\text{strongly in }L^2(0,T;L^2(\Omega))
\]
for some $\widetilde w\in L^2(0,T;H_0^1(\Omega))$.
We claim that $\widetilde w=w$. Indeed, by \eqref{eq:dal-difference},
\[
\|\widehat\nu_\tau-\nu(\bar h_\tau)\|_{L^2(0,T;H^{-1}(\Omega))}
\le
\tau\|\partial_t\widehat\nu_\tau\|_{L^2(0,T;H^{-1}(\Omega))}
\to 0.
\]
Hence $\widehat\nu_\tau$ and $\nu(\bar h_\tau)$ have the same limit in distributions, so
$\widetilde w=w$.

\begin{revision}
Set $\widetilde h:=\nu^{-1}(w)$. The global Lipschitz continuity of $\nu^{-1}$ and the strong
convergence of $\nu(\bar h_\tau)$ give
\[
\bar h_\tau=\nu^{-1}(\nu(\bar h_\tau))\to\widetilde h
\qquad\text{strongly in }L^2(0,T;L^2(\Omega)).
\]
On the other hand, \eqref{eq:cv-hbar-weak} implies weak convergence of $\bar h_\tau$ to the
previously extracted function $h$ in $L^2(0,T;L^2(\Omega))$. Uniqueness of the weak limit therefore
gives $\widetilde h=h$, and hence $w=\nu(h)$. Together with the preceding convergences, this
proves \eqref{eq:cv-nuhat-weak}--\eqref{eq:cv-nuhat-strong}, \eqref{eq:cv-nubar-strong}, and
\eqref{eq:cv-hbar-strong}.

For completeness, the weak-* identification in \eqref{eq:cv-hbar-star} uses the uniform
$L^\infty(0,T;L^2(\Omega))$ bound in addition to strong $L^2$ convergence. Given
$\phi\in L^1(0,T;L^2(\Omega))$, choose $\phi_k\in L^2(0,T;L^2(\Omega))$ with
$\phi_k\to\phi$ in $L^1(0,T;L^2(\Omega))$. Then
\[
\left|\int_0^T(\bar h_\tau-h,\phi-\phi_k)\,dt\right|
\le
\bigl(\|\bar h_\tau\|_{L^\infty L^2}+\|h\|_{L^\infty L^2}\bigr)
\|\phi-\phi_k\|_{L^1L^2}.
\]
For fixed $k$, the integral against $\phi_k$ tends to zero by strong $L^2$ convergence; letting
$k\to\infty$ proves convergence against every $L^1(0,T;L^2(\Omega))$ test and hence the asserted
weak-* limit.
\end{revision}

It remains to identify the limit of $\bar v_\tau$. From the definition of the affine interpolant,
for $t\in(t^{n-1},t^n]$ one has
\[
\bar v_\tau(t)-\widehat v_\tau(t)
=
\frac{t^n-t}{\tau}(v^n-v^{n-1})
=
(t^n-t)\,\partial_t \widehat v_\tau(t).
\]
Hence
\[
\norm{\bar v_\tau-\widehat v_\tau}_{L^2(0,T;L^2(\Omega)^d)}^2
\le
\tau^2 \norm{\partial_t \widehat v_\tau}_{L^2(0,T;L^2(\Omega)^d)}^2
\le
C\tau^2,
\]
and therefore
\[
\bar v_\tau-\widehat v_\tau \to 0
\qquad\text{strongly in }L^2(0,T;L^2(\Omega)^d).
\]
Combined with \eqref{eq:cv-vhat}, this proves \eqref{eq:cv-vbar}.
\begin{revision}
The weak-* limit in \eqref{eq:cv-vbar-star} is also $v$: approximate an arbitrary test in
$L^1(0,T;L^2(\Omega)^d)$ by $L^2(0,T;L^2(\Omega)^d)$ tests and use the uniform
$L^\infty(0,T;L^2(\Omega)^d)$ bound exactly as in the preceding argument.
\end{revision}

Finally, \eqref{eq:cv-vhat-strong} follows from an Arzelà--Ascoli argument. Indeed, the sequence
$(\widehat v_\tau)_\tau$ is bounded in $C([0,T];L^2(\Omega)^d)$ by the continuous embedding
$H^1(0,T;L^2(\Omega)^d)\hookrightarrow C([0,T];L^2(\Omega)^d)$, and hence bounded in
$C([0,T];H^{-1}(\Omega)^d)$. Moreover,
\[
\norm{\widehat v_\tau(t)-\widehat v_\tau(s)}_{H^{-1}(\Omega)^d}
\le
\norm{\widehat v_\tau(t)-\widehat v_\tau(s)}_{L^2(\Omega)^d}
\le
|t-s|^{1/2}\norm{\partial_t\widehat v_\tau}_{L^2(0,T;L^2(\Omega)^d)},
\]
so the family is equicontinuous in $H^{-1}(\Omega)^d$. Since the embedding
$L^2(\Omega)^d\hookrightarrow H^{-1}(\Omega)^d$ is compact, pointwise relative compactness holds,
\begin{revision}
and Arzelà--Ascoli therefore yields, after extraction of a further subsequence, strong convergence in
$C([0,T];H^{-1}(\Omega)^d)$ to some limit $\widetilde v$. The weak convergence
\eqref{eq:cv-vhat} shows that the same sequence converges to $v$ in distributions; hence
$\widetilde v=v$. This proves \eqref{eq:cv-vhat-strong}.
\end{revision}
\end{proof}

\subsection{Passage to the limit}

\begin{theorem}[Existence for the cutoff system: the case $h_0\in H_0^1(\Omega)$]
\label{thm:existence-cutoff-H1}
Let $\Omega\subset\R^d$, $d\le3$, be a bounded Lipschitz domain, let $T>0$, and let $a,b>0$\rev{.}
\rev{Fix $M\in(0,1/a)$ in the cutoff definitions above.}
For every initial data
\[
v_0\in L^2(\Omega)^d,
\qquad
h_0\in H_0^1(\Omega),
\]
there exists a weak solution $(v,h)$ of the cutoff Westervelt system
\eqref{eq:cutoff-westervelt} in the sense of Definition~\ref{def:weak-solution}. Moreover,
\begin{equation}\label{eq:apriori-limit-H1}
\sup_{t\in[0,T]} \norm{v(t)}_{L^2(\Omega)}^2
+ \sup_{t\in[0,T]} \norm{h(t)}_{L^2(\Omega)}^2
+ \norm{\nabla h}_{L^2(0,T;L^2(\Omega))}^2
\le C\Bigl(\norm{v_0}_{L^2(\Omega)}^2 + \norm{h_0}_{L^2(\Omega)}^2\Bigr).
\end{equation}
\rev{The constant $C$ may depend on $\Omega,T,a,b,$ and $M$ through $\delta$ and $\Lambda$, but it is independent of the Rothe time step.}
\end{theorem}

\begin{remark}[On uniqueness]
Theorem~\ref{thm:existence-cutoff-H1} provides a low-order existence result for the cutoff problem.
Full weak uniqueness is not a direct consequence of the energy estimate. A weak--strong uniqueness
principle is proved in Section~\ref{sec:weak-strong-uniqueness}; it applies whenever one solution
has enough time regularity to serve as a reference state.
\end{remark}

\begin{proof}
By Lemma~\ref{lem:compactness}, after extraction of a subsequence there exist
\[
v\in H^1(0,T;L^2(\Omega)^d)\cap L^\infty(0,T;L^2(\Omega)^d),
\qquad
h\in L^\infty(0,T;L^2(\Omega))\cap L^2(0,T;H_0^1(\Omega)),
\]
such that the convergences \eqref{eq:cv-vhat}--\eqref{eq:cv-hbar-strong} hold.

\smallskip
\noindent
\textbf{Step 1: Passage to the limit in the first equation.}
For almost every $t\in(0,T)$, the first discrete relation in \eqref{eq:rothe-system} can be written as
\[
\ip{\partial_t\widehat v_\tau(t)}{\varphi}
+
\ip{\nabla \bar h_\tau(t)}{\varphi}
=0
\qquad\forall\,\varphi\in L^2(\Omega)^d.
\]
Integrating over $(0,T)$ against an arbitrary test function
$\zeta\in L^2(0,T;L^2(\Omega)^d)$ yields
\[
\int_0^T \ip{\partial_t\widehat v_\tau(t)}{\zeta(t)}\,\ds
+
\int_0^T \ip{\nabla \bar h_\tau(t)}{\zeta(t)}\,\ds
=0.
\]
Passing to the limit by \eqref{eq:cv-vhat} and \eqref{eq:cv-hbar-weak}, we obtain
\[
\int_0^T \ip{\partial_t v(t)}{\zeta(t)}\,\ds
+
\int_0^T \ip{\nabla h(t)}{\zeta(t)}\,\ds
=0
\qquad\forall\,\zeta\in L^2(0,T;L^2(\Omega)^d).
\]
Hence
\[
\partial_t v + \nabla h =0
\qquad\text{in }L^2(0,T;L^2(\Omega)^d),
\]
which is equivalent to \eqref{eq:weak-v}.

\smallskip
\noindent
\textbf{Step 2: Passage to the limit in the second equation.}
For almost every $t\in(0,T)$, the second discrete relation in \eqref{eq:rothe-system} reads
\[
\dual{\partial_t\widehat\nu_\tau(t)}{\xi}
-
\ip{\bar v_\tau(t)}{\nabla\xi}
+
b\ip{\nabla\bar h_\tau(t)}{\nabla\xi}
=0
\qquad\forall\,\xi\in H_0^1(\Omega).
\]
Let $\eta\in L^2(0,T;H_0^1(\Omega))$ be arbitrary. Integrating in time gives
\[
\int_0^T \dual{\partial_t\widehat\nu_\tau(t)}{\eta(t)}\,\ds
-
\int_0^T \ip{\bar v_\tau(t)}{\nabla\eta(t)}\,\ds
+
b\int_0^T \ip{\nabla\bar h_\tau(t)}{\nabla\eta(t)}\,\ds
=0.
\]
Passing to the limit by \eqref{eq:cv-nuhat-time}, \eqref{eq:cv-vbar}, and \eqref{eq:cv-hbar-weak},
we arrive at
\[
\int_0^T \dual{\partial_t\nu(h)(t)}{\eta(t)}\,\ds
-
\int_0^T \ip{v(t)}{\nabla\eta(t)}\,\ds
+
b\int_0^T \ip{\nabla h(t)}{\nabla\eta(t)}\,\ds
=0
\qquad\forall\,\eta\in L^2(0,T;H_0^1(\Omega)).
\]
Thus
\[
\partial_t \nu(h) + \diver v = b\Delta h
\qquad\text{in }L^2(0,T;H^{-1}(\Omega)),
\]
which is equivalent to \eqref{eq:weak-h}.

\smallskip
\noindent
\textbf{Step 3: Initial conditions.}
Since the evaluation map
\(
w\mapsto w(0)
\)
is continuous from $H^1(0,T;L^2(\Omega)^d)$ to $L^2(\Omega)^d$, the weak convergence
\eqref{eq:cv-vhat} implies
\[
\widehat v_\tau(0)\rightharpoonup v(0)
\qquad\text{weakly in }L^2(\Omega)^d.
\]
Because $\widehat v_\tau(0)=v_0$ for every $\tau$, it follows that
\[
v(0)=v_0
\qquad\text{in }L^2(\Omega)^d.
\]

Likewise, the evaluation map is continuous from
$H^1(0,T;H^{-1}(\Omega))$ to $H^{-1}(\Omega)$. Hence
\eqref{eq:cv-nuhat-time} yields
\[
\widehat\nu_\tau(0)\rightharpoonup \nu(h)(0)
\qquad\text{weakly in }H^{-1}(\Omega).
\]
Since $\widehat\nu_\tau(0)=\nu(h_0)$ for every $\tau$, we conclude that
\[
\nu(h)(0)=\nu(h_0)
\qquad\text{in }H^{-1}(\Omega).
\]

\smallskip
\noindent
\textbf{Step 4: A priori estimate.}
By \eqref{eq:cv-vbar-star}, \eqref{eq:cv-hbar-star}, and \eqref{eq:cv-hbar-weak}, together with the
uniform bounds \eqref{eq:uniform-bounds-discrete}, weak lower semicontinuity yields
\[
\norm{v}_{L^\infty(0,T;L^2(\Omega)^d)}^2
+
\norm{h}_{L^\infty(0,T;L^2(\Omega))}^2
+
\norm{\nabla h}_{L^2(0,T;L^2(\Omega))}^2
\le
C\Bigl(\norm{v_0}_{L^2(\Omega)}^2+\norm{h_0}_{L^2(\Omega)}^2\Bigr).
\]
This estimate is first obtained with the essential supremum in time. Since
$v\in H^1(0,T;L^2(\Omega)^d)\hookrightarrow C([0,T];L^2(\Omega)^d)$ and, by
Remark~\ref{rem:time-continuity-nu}, $h\in C([0,T];L^2(\Omega))$, the essential supremum agrees
with the supremum over these continuous representatives. Thus the estimate may be written as
\eqref{eq:apriori-limit-H1}. We have thus shown that $(v,h)$ is a weak solution
of \eqref{eq:cutoff-westervelt} in the sense of Definition~\ref{def:weak-solution}.
\end{proof}

\begin{theorem}[Existence for the cutoff system: the case $h_0\in L^2(\Omega)$]
\label{thm:existence-cutoff}
Let $\Omega\subset\R^d$, $d\le3$, be a bounded Lipschitz domain, let $T>0$, and let $a,b>0$\rev{.}
\rev{Fix $M\in(0,1/a)$ in the cutoff definitions above.}
For every initial data
\[
v_0\in L^2(\Omega)^d,
\qquad
h_0\in L^2(\Omega),
\]
there exists a weak solution $(v,h)$ of the cutoff Westervelt system \eqref{eq:cutoff-westervelt}
in the sense of Definition~\ref{def:weak-solution}. Moreover,
\begin{equation}\label{eq:apriori-limit-L2}
\sup_{t\in[0,T]} \norm{v(t)}_{L^2(\Omega)}^2
+ \sup_{t\in[0,T]} \norm{h(t)}_{L^2(\Omega)}^2
+ \norm{\nabla h}_{L^2(0,T;L^2(\Omega))}^2
\le C\Bigl(\norm{v_0}_{L^2(\Omega)}^2 + \norm{h_0}_{L^2(\Omega)}^2\Bigr).
\end{equation}
\rev{The constant $C$ may depend on $\Omega,T,a,b,$ and $M$ through $\delta$ and $\Lambda$, but it is independent of the approximating index used below.}
\end{theorem}

\begin{remark}[\rev{Scope of the $L^2$-data construction}]
\begin{revision}
Theorem~\ref{thm:existence-cutoff} proves existence for general $L^2$ initial enthalpy by
approximating the initial datum with $H_0^1$ data and then passing to a limit in the corresponding
continuous weak solutions. The compactness argument above establishes direct convergence of the
Rothe interpolants only when the initial enthalpy belongs to $H_0^1(\Omega)$. No claim is made here
that the Rothe sequence initialized directly with an arbitrary $h_0\in L^2(\Omega)$ converges.
\end{revision}
\end{remark}

\begin{proof}
Choose a sequence $(h_0^m)_{m\in\N}\subset H_0^1(\Omega)$ such that
\[
h_0^m\to h_0
\qquad\text{strongly in }L^2(\Omega).
\]
Since $\nu$ is globally Lipschitz, we also have
\[
\nu(h_0^m)\to \nu(h_0)
\qquad\text{strongly in }L^2(\Omega)\hookrightarrow H^{-1}(\Omega).
\]

For each $m\in\N$, Theorem~\ref{thm:existence-cutoff-H1} yields a weak solution
$(v_m,h_m)$ corresponding to the initial data $(v_0,h_0^m)$. Moreover,
\begin{equation}\label{eq:apriori-m}
\sup_{t\in[0,T]} \norm{v_m(t)}_{L^2(\Omega)}^2
+ \sup_{t\in[0,T]} \norm{h_m(t)}_{L^2(\Omega)}^2
+ \norm{\nabla h_m}_{L^2(0,T;L^2(\Omega))}^2
\le C\Bigl(\norm{v_0}_{L^2(\Omega)}^2 + \norm{h_0^m}_{L^2(\Omega)}^2\Bigr).
\end{equation}
Since $(h_0^m)$ is bounded in $L^2(\Omega)$, the right-hand side is bounded uniformly in $m$.

We next derive a uniform bound for $\partial_t \nu(h_m)$ in $L^2(0,T;H^{-1}(\Omega))$.
Indeed, from the weak formulation \eqref{eq:weak-h} for $(v_m,h_m)$ we obtain for a.e.\ $t\in(0,T)$
and every $\xi\in H_0^1(\Omega)$,
\[
\dual{\partial_t\nu(h_m)(t)}{\xi}
=
\ip{v_m(t)}{\nabla\xi}
-
b\ip{\nabla h_m(t)}{\nabla\xi}.
\]
Hence
\[
\norm{\partial_t\nu(h_m)(t)}_{H^{-1}(\Omega)}
\le
\norm{v_m(t)}_{L^2(\Omega)^d}
+
b\norm{\nabla h_m(t)}_{L^2(\Omega)},
\]
and therefore, by \eqref{eq:apriori-m},
\begin{equation}\label{eq:dtnu-m}
\norm{\partial_t\nu(h_m)}_{L^2(0,T;H^{-1}(\Omega))}\le C.
\end{equation}

Since $\nu$ is globally Lipschitz and $\nu(0)=0$, the composition $\nu(h_m)$ belongs to
$L^2(0,T;H_0^1(\Omega))$ and satisfies, by the Sobolev chain rule and
$\nabla\nu(h_m)=m(h_m)\nabla h_m$ together with \eqref{eq:mobility-bounds},
\[
\norm{\nabla\nu(h_m)}_{L^2(0,T;L^2(\Omega))}
\le \Lambda \norm{\nabla h_m}_{L^2(0,T;L^2(\Omega))}\le C
\]
uniformly in $m$. This bound uses only the uniform estimate \eqref{eq:apriori-m} on
$\nabla h_m$ and is independent of the $H_0^1$-norm of the approximating data $h_0^m$.
Together with \eqref{eq:dtnu-m}, it shows that $(\nu(h_m))_{m\in\N}$ is bounded in
\[
L^2(0,T;H_0^1(\Omega))\cap H^1(0,T;H^{-1}(\Omega)).
\]
Hence, after extraction of a subsequence, there exists
\[
w\in L^2(0,T;H_0^1(\Omega))\cap H^1(0,T;H^{-1}(\Omega))
\]
such that
\[
\nu(h_m)\rightharpoonup w
\quad\text{weakly in }L^2(0,T;H_0^1(\Omega)),
\qquad
\nu(h_m)\rightharpoonup w
\quad\text{weakly in }H^1(0,T;H^{-1}(\Omega)).
\]
By the Aubin--Lions--Simon compactness theorem \cite{Simon1987}, after extraction of a further subsequence,
\begin{equation}\label{eq:nuhm-strong}
\nu(h_m)\to w
\qquad\text{strongly in }L^2(0,T;L^2(\Omega)).
\end{equation}

The first equation in the weak formulation for $(v_m,h_m)$ gives
\[
\partial_t v_m=-\nabla h_m
\qquad\text{in }L^2(0,T;L^2(\Omega)^d).
\]
Together with \eqref{eq:apriori-m}, this yields
\[
\|v_m\|_{H^1(0,T;L^2(\Omega)^d)}\le C
\]
uniformly in $m$. Therefore, after extracting a further subsequence,
\[
v_m \stackrel{*}{\rightharpoonup} v
\quad\text{in }L^\infty(0,T;L^2(\Omega)^d),
\qquad
v_m \rightharpoonup v
\quad\text{in }H^1(0,T;L^2(\Omega)^d),
\]
and, by \eqref{eq:apriori-m},
\[
h_m \stackrel{*}{\rightharpoonup} h
\quad\text{in }L^\infty(0,T;L^2(\Omega)),
\qquad
h_m \rightharpoonup h
\quad\text{in }L^2(0,T;H_0^1(\Omega)).
\]
Since $\nu^{-1}$ is globally Lipschitz by \eqref{eq:nu-bounds}, the strong convergence
\eqref{eq:nuhm-strong} implies
\[
h_m=\nu^{-1}(\nu(h_m))\to \nu^{-1}(w)
\qquad\text{strongly in }L^2(0,T;L^2(\Omega)).
\]
The weak limit of $h_m$ in $L^2(0,T;H_0^1(\Omega))$ is therefore
$h=\nu^{-1}(w)$, and consequently $w=\nu(h)$. Thus
\[
\nu(h_m)\rightharpoonup \nu(h)
\quad\text{weakly in }L^2(0,T;H_0^1(\Omega)),
\qquad
\nu(h_m)\rightharpoonup \nu(h)
\quad\text{weakly in }H^1(0,T;H^{-1}(\Omega)).
\]
This also shows that the limiting solution satisfies the mobility-variable regularity required in
Definition~\ref{def:weak-solution}.

We now pass to the limit in the weak formulations for $(v_m,h_m)$.
The first relation \eqref{eq:weak-v} passes to the limit by weak convergence of
$\partial_t v_m$ in $L^2(0,T;L^2(\Omega)^d)$ and weak convergence of $\nabla h_m$ in
$L^2(0,T;L^2(\Omega)^d)$.
The second relation \eqref{eq:weak-h} passes to the limit by weak convergence of
$\partial_t\nu(h_m)$ in $L^2(0,T;H^{-1}(\Omega))$, weak convergence of $v_m$ in
$L^2(0,T;L^2(\Omega)^d)$, and weak convergence of $\nabla h_m$ in
$L^2(0,T;L^2(\Omega)^d)$.
Hence $(v,h)$ solves \eqref{eq:weak-v}--\eqref{eq:weak-h}.

It remains to identify the initial values.
Since $v_m(0)=v_0$ for every $m$ and
\[
v_m \rightharpoonup v \qquad\text{in }H^1(0,T;L^2(\Omega)^d),
\]
continuity of the trace operator from $H^1(0,T;L^2(\Omega)^d)$ to $L^2(\Omega)^d$ yields
\[
v(0)=v_0
\qquad\text{in }L^2(\Omega)^d.
\]
Likewise, from
\[
\nu(h_m)(0)=\nu(h_0^m),
\qquad
\nu(h_m)\rightharpoonup \nu(h)\quad\text{in }H^1(0,T;H^{-1}(\Omega)),
\]
and
\[
\nu(h_0^m)\to \nu(h_0)\quad\text{in }H^{-1}(\Omega),
\]
we obtain by continuity of the trace map on $H^1(0,T;H^{-1}(\Omega))$ that
\[
\nu(h)(0)=\nu(h_0)
\qquad\text{in }H^{-1}(\Omega).
\]

Finally, \eqref{eq:apriori-limit-L2} follows from \eqref{eq:apriori-m} by weak lower semicontinuity,
first with the essential supremum in time. Since
$v\in H^1(0,T;L^2(\Omega)^d)\hookrightarrow C([0,T];L^2(\Omega)^d)$ and
$h\in C([0,T];L^2(\Omega))$ by Remark~\ref{rem:time-continuity-nu}, the estimate may be written
with $\sup_{t\in[0,T]}$ for the continuous representatives.
\end{proof}

\section{Weak--strong uniqueness for the cutoff system}\label{sec:weak-strong-uniqueness}

The low-order energy estimate used in the existence proof does not yield full weak uniqueness. The
obstruction is the nonlinear time derivative $\partial_t\nu(h)$: after subtracting two weak
solutions, the term
\[
\dual{\partial_t(\nu(h_1)-\nu(h_2))}{h_1-h_2}
\]
is not the derivative of a symmetric quadratic distance. Uniqueness can nevertheless be recovered
relative to a sufficiently regular reference solution by using the relative entropy associated with
the inverse map $\nu^{-1}$.

Throughout this section we set
\[
\beta:=\nu^{-1},
\qquad
\Psi(s):=\int_0^s\beta(\rho)\,d\rho.
\]
For a reference enthalpy $H$ and another enthalpy $h$, define
\begin{equation}\label{eq:relative-entropy-def}
\mathscr R(h\mid H)
:=
\int_\Omega
\Bigl[
\Psi(\nu(h))-\Psi(\nu(H))-H\bigl(\nu(h)-\nu(H)\bigr)
\Bigr]\,dx.
\end{equation}
Equivalently,
\begin{equation}\label{eq:relative-entropy-integral}
\mathscr R(h\mid H)
=
\int_\Omega\int_H^h (s-H)m(s)\,ds\,dx,
\end{equation}
and therefore
\begin{equation}\label{eq:relative-entropy-equivalence}
\frac{\delta}{2}\norm{h-H}_{L^2(\Omega)}^2
\le
\mathscr R(h\mid H)
\le
\frac{\Lambda}{2}\norm{h-H}_{L^2(\Omega)}^2.
\end{equation}

\begin{revision}
\begin{lemma}[Relative chain rule]\label{lem:relative-chain-rule}
Let
\[
w,W\in L^2(0,T;H_0^1(\Omega))\cap H^1(0,T;H^{-1}(\Omega)),
\]
set
\[
h:=\beta(w),\qquad H:=\beta(W),\qquad q:=w-W,\qquad r:=h-H,
\]
and assume in addition that
\[
H\in W^{1,1}(0,T;L^\infty(\Omega)).
\]
Then $W=\nu(H)$ belongs to $W^{1,1}(0,T;L^\infty(\Omega))$ and
\[
W_t=m(H)H_t
\qquad\text{in }L^1(0,T;L^\infty(\Omega)).
\]
The map
\[
t\longmapsto \mathscr R(h(t)\mid H(t))
\]
is absolutely continuous on $[0,T]$, and for a.e. $t\in(0,T)$,
\begin{equation}\label{eq:relative-chain-rule}
\dual{\partial_t q(t)}{r(t)}
=
\frac{d}{dt}\mathscr R(h(t)\mid H(t))
+
\int_\Omega H_t(t)\bigl(q(t)-m(H(t))r(t)\bigr)\,dx.
\end{equation}
All duality pairings in \eqref{eq:relative-chain-rule} are between $H^{-1}(\Omega)$ and
$H_0^1(\Omega)$.
\end{lemma}

\begin{proof}
Set
\[
V:=H_0^1(\Omega),\qquad \mathcal H:=L^2(\Omega),\qquad V^*:=H^{-1}(\Omega),
\]
and define the convex functional
\[
\mathcal J(u):=\int_\Omega\Psi(u)\,dx,
\qquad \Psi'(s)=\beta(s).
\]
We first establish the chain rule
\begin{equation}\label{eq:convex-gelfand-chain}
\frac{d}{dt}\mathcal J(u(t))=\dual{u_t(t)}{\beta(u(t))}
\qquad\text{for a.e. }t\in(0,T)
\end{equation}
for every $u\in L^2(0,T;V)\cap H^1(0,T;V^*)$.
Because $\beta$ is globally Lipschitz and $\beta(0)=0$, the Sobolev chain rule gives
$\beta(u)\in L^2(0,T;V)$. Moreover, the Lions--Magenes embedding yields
$u\in C([0,T];\mathcal H)$, and the quadratic growth of $\Psi$ implies that
$t\mapsto\mathcal J(u(t))$ is continuous.

For $s>0$ and a.e. $t\in(0,T-s)$, convexity of $\Psi$ gives
\[
\dual{u(t+s)-u(t)}{\beta(u(t))}
\le \mathcal J(u(t+s))-\mathcal J(u(t))
\le \dual{u(t+s)-u(t)}{\beta(u(t+s))}.
\]
Multiply by a nonnegative $\varphi\in C_c^\infty(0,T)$, divide by $s$, and integrate in time.
As $s\downarrow0$, the difference quotients of $u$ converge to $u_t$ in $L^2(0,T;V^*)$,
while time translations of $\beta(u)\in L^2(0,T;V)$ converge strongly in that space. Both the lower
and upper bounds therefore converge to
\[
\int_0^T\varphi(t)\dual{u_t(t)}{\beta(u(t))}\,dt.
\]
The difference quotient of the continuous scalar function $\mathcal J(u(\cdot))$ converges in the
distributional sense to its distributional derivative. Hence
\[
-\int_0^T\mathcal J(u(t))\varphi'(t)\,dt
=
\int_0^T\varphi(t)\dual{u_t(t)}{\beta(u(t))}\,dt.
\]
The right-hand side belongs to $L^1(0,T)$, so $\mathcal J(u)\in W^{1,1}(0,T)$ and
\eqref{eq:convex-gelfand-chain} follows.

Applying \eqref{eq:convex-gelfand-chain} first to $w$ and then to $W$ gives
\begin{equation}\label{eq:Jw-JW-chain}
\frac{d}{dt}\mathcal J(w)=\dual{w_t}{h},
\qquad
\frac{d}{dt}\mathcal J(W)=\dual{W_t}{H}
\quad\text{in }L^1(0,T).
\end{equation}
Since $H\in W^{1,1}(0,T;L^\infty(\Omega))$, the ordinary Bochner chain rule for the
$C^1$ function $\nu$ with bounded derivative yields
\[
W=\nu(H)\in W^{1,1}(0,T;L^\infty(\Omega)),
\qquad W_t=m(H)H_t.
\]
This derivative agrees with the $H^{-1}$ derivative already supplied by the assumed
$H^1(0,T;H^{-1})$ regularity of $W$.

We also need the product rule
\begin{equation}\label{eq:mixed-product-rule}
\frac{d}{dt}\int_\Omega Hq\,dx
=
\dual{q_t}{H}+\int_\Omega H_tq\,dx
\quad\text{in }L^1(0,T).
\end{equation}
To justify it, mollify $q$ and $H$ in time on compact subintervals of $(0,T)$. The classical
product rule holds for the mollified functions. In the first term, passage to the limit follows from
$q_t\in L^2(0,T;V^*)$ and $H\in L^2(0,T;V)$. In the second term, it follows from
$q\in C([0,T];\mathcal H)$, $H_t\in L^1(0,T;L^\infty(\Omega))$, and the corresponding convergence
of the time mollifications. Since both $q$ and $H$ have continuous representatives in
$L^2(\Omega)$, the identity extends to the endpoints and proves \eqref{eq:mixed-product-rule}.

Using
\[
\mathscr R(h\mid H)=\mathcal J(w)-\mathcal J(W)-\int_\Omega Hq\,dx,
\]
relations \eqref{eq:Jw-JW-chain} and \eqref{eq:mixed-product-rule} imply
\[
\frac{d}{dt}\mathscr R(h\mid H)
=\dual{w_t}{h-H}-\int_\Omega H_tq\,dx.
\]
Finally, $q_t=w_t-W_t$ and $W_t=m(H)H_t$, so
\[
\dual{q_t}{r}
=\dual{w_t}{h-H}-\int_\Omega m(H)H_t(h-H)\,dx.
\]
Combining the last two identities gives \eqref{eq:relative-chain-rule}. The preceding chain and
product rules also show that $t\mapsto\mathscr R(h(t)\mid H(t))$ is absolutely continuous on
$[0,T]$.
\end{proof}
\end{revision}

\begin{theorem}[Weak--strong uniqueness for the cutoff system]
\label{thm:weak-strong-uniqueness-cutoff}
Let $(v,h)$ and $(V,H)$ be weak solutions of the cutoff system
\eqref{eq:cutoff-westervelt} in the sense of Definition~\ref{def:weak-solution}. Assume that the
reference solution satisfies
\[
\rev{H\in W^{1,1}(0,T;L^\infty(\Omega)).}
\]
Then, for every $t\in[0,T]$,
\begin{align}\label{eq:weak-strong-stability}
&\frac12\norm{v(t)-V(t)}_{L^2(\Omega)}^2+
\mathscr R(h(t)\mid H(t))
\le
\left(
\frac12\norm{v(0)-V(0)}_{L^2(\Omega)}^2+
\mathscr R(h(0)\mid H(0))
\right)
\exp\left(
\frac{a}{\delta}\int_0^t\norm{H_t(s)}_{L^\infty(\Omega)}\,ds
\right).
\end{align}
In particular, if the initial data coincide, then
\[
v=V,
\qquad
h=H
\quad\text{a.e. in }\Omega\times(0,T).
\]
\end{theorem}

\begin{proof}
Set
\[
z:=v-V,
\qquad
r:=h-H,
\qquad
q:=\nu(h)-\nu(H).
\]
Subtracting the two weak formulations gives
\[
\partial_t z+\nabla r=0
\qquad\text{in }L^2(0,T;L^2(\Omega)^d),
\]
and
\[
\partial_t q+\diver z=b\Delta r
\qquad\text{in }L^2(0,T;H^{-1}(\Omega)).
\]
Testing the first equation by $z$ yields
\[
\frac12\frac{d}{dt}\norm{z}_{L^2(\Omega)}^2+\ip{\nabla r}{z}=0
\]
for a.e.\ $t\in(0,T)$. Testing the second equation by $r\in H_0^1(\Omega)$ gives
\[
\dual{\partial_t q}{r}-\ip{z}{\nabla r}+b\norm{\nabla r}_{L^2(\Omega)}^2=0.
\]
Adding the two identities and using Lemma~\ref{lem:relative-chain-rule}, we obtain
\begin{align*}
&\frac{d}{dt}\left(
\frac12\norm{z}_{L^2(\Omega)}^2+\mathscr R(h\mid H)
\right)
+b\norm{\nabla r}_{L^2(\Omega)}^2
=
-\int_\Omega H_t\bigl(q-m(H)r\bigr)\,dx .
\end{align*}
Since \(m=1-a\TM\) is globally Lipschitz with Lipschitz constant at most \(a\), Taylor's formula with integral remainder gives, for a.e. \(x\),
\[
\abs{q-m(H)r}
=
\abs{\nu(h)-\nu(H)-m(H)(h-H)}
\le
\frac a2\abs{h-H}^2.
\]
Using \eqref{eq:relative-entropy-equivalence}, this gives
\[
\abs{\int_\Omega H_t\bigl(q-m(H)r\bigr)\,dx}
\le
\frac a2\norm{H_t}_{L^\infty(\Omega)}\norm{r}_{L^2(\Omega)}^2
\le
\frac a\delta\norm{H_t}_{L^\infty(\Omega)}\mathscr R(h\mid H).
\]
Therefore, with
\[
\mathcal E(t):=\frac12\norm{z(t)}_{L^2(\Omega)}^2+\mathscr R(h(t)\mid H(t)),
\]
we have
\[
\frac{d}{dt}\mathcal E(t)
\le
\frac a\delta\norm{H_t(t)}_{L^\infty(\Omega)}\mathcal E(t)
\]
for a.e.\ $t\in(0,T)$. Gronwall's inequality yields \eqref{eq:weak-strong-stability}. If the
initial data coincide, then $\mathcal E(0)=0$, so $\mathcal E(t)=0$ for all $t\in[0,T]$. By
\eqref{eq:relative-entropy-equivalence}, $h=H$ a.e. in $\Omega\times(0,T)$. The equation
$\partial_t z=-\nabla r$ then gives $\partial_t z=0$, and since $z(0)=0$, we also obtain $v=V$.
\end{proof}

\section[A posteriori removal of the cutoff via a uniform bound]{A posteriori removal of the cutoff via an $L^\infty$ bound}\label{sec:cutoff-removal}

In this section we show how the cutoff can be removed \emph{a posteriori} once a suitable
$L^\infty$-control of the enthalpy is available. We first state the elementary criterion and then
formulate a compactness-based $H^2$ criterion for Rothe approximations. This keeps the cutoff
removal separate from the low-order existence theory: the latter yields weak solutions of the
cutoff problem, whereas recovering the uncut system requires additional control that is not
available from the basic energy estimate alone. We continue to work with the symmetric cutoff
\[
\TM(r)=\max\{-M,\min\{r,M\}\},
\qquad M\in(0,1/a),
\]
and the associated mobility and primitive
\[
m_M(h)=1-a\,\TM(h),
\qquad
\nu_M(h)=\int_0^h m_M(s)\,ds.
\]
We also write
\[
\nu_{\mathrm{unc}}(r):=r-\frac a2 r^2
\]
for the primitive corresponding to the uncut mobility $1-ar$.

\subsection{A simple cutoff-removal criterion}

\begin{lemma}[Cutoff inactive if the solution stays below the threshold]\label{lem:cutoff-inactive}
Let $(v,h)$ be a weak solution of the cutoff Westervelt system on $(0,T)$ in the sense of
Definition~\ref{def:weak-solution}. Assume, in addition, that
\[
\|h\|_{L^\infty(\Omega\times(0,T))} \le M.
\]
Then $\TM(h)=h$ a.e.\ in $\Omega\times(0,T)$, and consequently
\[
m_M(h)=1-ah,
\qquad
\nu_M(h)=\nu_{\mathrm{unc}}(h)
\quad\text{a.e. in }\Omega\times(0,T).
\]
Thus $(v,h)$ solves the \emph{uncut first-order} Westervelt system with primitive
$\nu_{\mathrm{unc}}$. If, in addition,
\[
v_0=-\nabla\psi_0
\qquad\text{for some }\psi_0\in H_0^1(\Omega),
\]
then Proposition~\ref{prop:potential-reconstruction} yields a scalar potential
$\psi$ solving the scalar Westervelt equation in conservative form.
\end{lemma}

\begin{proof}
If $\|h\|_{L^\infty(\Omega\times(0,T))}\le M$, then $|h(x,t)|\le M$ a.e., hence by definition
$\TM(h)=h$ a.e. in $\Omega\times(0,T)$. Inserting this identity into the weak formulation of the cutoff
problem yields the weak formulation of the uncut first-order system. The last statement follows from
Proposition~\ref{prop:potential-reconstruction}.
\end{proof}

\subsection{A conditional cutoff-removal criterion from a discrete $H^2$-bound}

The previous lemma reduces cutoff removal to proving a bound of the form
\[
\|h\|_{L^\infty(\Omega\times(0,T))}\le M.
\]
At the level of the low-order existence theory established above, such a bound is not automatic.
We therefore formulate a rigorous criterion which shows that cutoff removal follows once one can
derive a uniform discrete $H^2$-bound along a convergent family of Rothe approximations.

\begin{proposition}[Conditional cutoff removal from a uniform discrete $H^2$-bound]\label{prop:cutoff-removal-H2}
Assume $\Omega\subset\R^d$, $d\le 3$, has $C^{1,1}$ boundary. Let $(v,h)$ be a weak solution of
the cutoff Westervelt system on $(0,T)$ in the sense of Definition~\ref{def:weak-solution}.
Assume that there exist time steps $\tau_j\to0$ and corresponding Rothe interpolants
$\bar h_{\tau_j}$ such that
\begin{equation}\label{eq:H2-criterion-strong-L2}
\bar h_{\tau_j}\to h
\qquad\text{strongly in }L^2(0,T;L^2(\Omega)),
\end{equation}
and
\begin{equation}\label{eq:discrete-H2-bound}
\sup_{j\in\N}
\|\bar h_{\tau_j}\|_{L^\infty(0,T;H^2(\Omega)\cap H_0^1(\Omega))}
\le K
\end{equation}
for some constant $K>0$. Then
\(
h\in L^\infty(0,T;H^2(\Omega)\cap H_0^1(\Omega))
\)
and
\begin{equation}\label{eq:limit-H2-bound}
\|h\|_{L^\infty(0,T;H^2(\Omega))}\le K.
\end{equation}
Consequently, $\|h\|_{L^\infty(\Omega\times(0,T))}
\le C_{\rm emb}\,K$
where $C_{\rm emb}>0$ depends only on $\Omega$ and $d$.
In particular, if
\(
C_{\rm emb}\,K \le M,
\)
then the cutoff is inactive and $(v,h)$ solves the uncut first-order Westervelt system by
Lemma~\ref{lem:cutoff-inactive}. If, in addition, the compatibility
$v_0=-\nabla\psi_0$ holds for some $\psi_0\in H_0^1(\Omega)$, then
Proposition~\ref{prop:potential-reconstruction} yields the corresponding scalar Westervelt equation
in conservative form.
\end{proposition}

\begin{proof}
By \eqref{eq:discrete-H2-bound}, the family $(\bar h_{\tau_j})_{j\in\N}$ is bounded in
$L^\infty(0,T;H^2(\Omega)\cap H_0^1(\Omega))$.
\begin{revision}
The space $X:=H^2(\Omega)\cap H_0^1(\Omega)$ is a separable Hilbert space. Hence
$L^1(0,T;X^*)$ is separable, and Banach--Alaoglu gives weak-* \emph{sequential} compactness for
bounded sequences in the canonical dual realization of $L^\infty(0,T;X)$.
\end{revision}
\rev{Thus, after extraction of a subsequence, there exists}
\[
\widetilde h\in L^\infty(0,T;H^2(\Omega)\cap H_0^1(\Omega))
\]
such that
\[
\bar h_{\tau_j}\stackrel{*}{\rightharpoonup}\widetilde h
\qquad\text{in }L^\infty(0,T;H^2(\Omega)\cap H_0^1(\Omega)).
\]
Moreover,
\[
\|\widetilde h\|_{L^\infty(0,T;H^2(\Omega))}\le K.
\]
The strong convergence \eqref{eq:H2-criterion-strong-L2} implies that the same subsequence converges
to $h$ strongly in $L^2(0,T;L^2(\Omega))$. Therefore the weak-* limit $\widetilde h$ and the strong
$L^2$-limit $h$ coincide almost everywhere in $\Omega\times(0,T)$. Hence $\widetilde h=h$, which
proves
\[
h\in L^\infty(0,T;H^2(\Omega)\cap H_0^1(\Omega))
\]
and \eqref{eq:limit-H2-bound}.
Finally, since $H^2(\Omega)\hookrightarrow L^\infty(\Omega)$ continuously for $d\le 3$, we obtain
\(
\|h\|_{L^\infty(\Omega\times(0,T))}
\le C_{\rm emb}\,K.
\)
The final assertion then follows from
Lemma~\ref{lem:cutoff-inactive} and Proposition~\ref{prop:potential-reconstruction}.
\end{proof}

\begin{remark}[\rev{One-dimensional alternative criterion}]
\label{rem:one-dimensional-small-data-route}
\begin{revision}
On $\Omega=(0,1)$, a uniform-in-time $H_0^1$ bound is already sufficient for cutoff removal. More
precisely, suppose that a convergent family of Rothe interpolants satisfies
\[
\bar h_{\tau_j}\to h
\qquad\text{strongly in }L^2(0,T;L^2(0,1)),
\]
and
\[
\sup_j\|\bar h_{\tau_j}\|_{L^\infty(0,T;H_0^1(0,1))}\le K_1.
\]
The one-dimensional embedding $H_0^1(0,1)\hookrightarrow L^\infty(0,1)$ gives
\[
\|\bar h_{\tau_j}\|_{L^\infty((0,T)\times(0,1))}\le C_{1D}K_1.
\]
After extraction of a subsequence, the strong $L^2$ convergence is pointwise almost everywhere;
therefore $|h|\le C_{1D}K_1$ almost everywhere. If $C_{1D}K_1\le M$, Lemma~\ref{lem:cutoff-inactive}
applies. The basic energy estimate supplies only an $L^2(0,T;H_0^1)$ bound, not the required
$L^\infty(0,T;H_0^1)$ bound. Establishing the latter by a higher-order small-data argument is beyond
the scope of this paper, and no one-dimensional small-data theorem is claimed here.
\end{revision}
\end{remark}

\section{Consistency analysis for the forced cutoff problem}\label{sec:error-estimate}

In order to derive a quantitative estimate for the Rothe approximation, it is convenient to include
forcing terms and consider
\begin{equation}\label{eq:cutoff-westervelt-forced}
\begin{aligned}
\dt v + \nabla h &= f_v &&\text{in } (0,T)\times\Omega,\\
\dt \nu(h) + \diver v &= b\Delta h + f_h &&\text{in } (0,T)\times\Omega,\\
h&=0 &&\text{on } (0,T)\times\partial\Omega,\\
v(0)=v_0,\qquad h(0)&=h_0 &&\text{in } \Omega.
\end{aligned}
\end{equation}
Here $f_v\in L^2(0,T;L^2(\Omega)^d)$ and $f_h\in L^2(0,T;H^{-1}(\Omega))$ are given.
For the time discretization we use the averages
\begin{equation}\label{eq:def-forcing-averages}
f_v^n := \frac1\tau\int_{t^{n-1}}^{t^n} f_v(s)\,\ds,
\qquad
f_h^n := \frac1\tau\int_{t^{n-1}}^{t^n} f_h(s)\,\ds.
\end{equation}
The Rothe scheme now reads: given $(v^{n-1},h^{n-1})$, find $(v^n,h^n)$ such that
\begin{equation}\label{eq:rothe-system-forced}
\begin{aligned}
\frac{1}{\tau}\ip{v^n-v^{n-1}}{\varphi} + \ip{\nabla h^n}{\varphi} &= \ip{f_v^n}{\varphi}
&&\forall\,\varphi\in L^2(\Omega)^d,\\
\frac{1}{\tau}\ip{\nu(h^n)-\nu(h^{n-1})}{\xi} - \ip{v^n}{\nabla\xi} + b\ip{\nabla h^n}{\nabla\xi} &= \dual{f_h^n}{\xi}
&&\forall\,\xi\in H_0^1(\Omega).
\end{aligned}
\end{equation}

\begin{proposition}[Existence of one forced time step]\label{prop:forced-step}
For every $n=1,\dots,N$, every
\[
(v^{n-1},h^{n-1})\in L^2(\Omega)^d\times L^2(\Omega),
\]
and every
\[
f_v^n\in L^2(\Omega)^d,
\qquad
f_h^n\in H^{-1}(\Omega),
\]
the forced Rothe step \eqref{eq:rothe-system-forced} admits a unique solution
\[
h^n\in H_0^1(\Omega),
\qquad
v^n\in L^2(\Omega)^d.
\]
\end{proposition}

\begin{proof}
Eliminating $v^n$ from the first equation gives
\[
v^n=v^{n-1}-\tau\nabla h^n+\tau f_v^n.
\]
Inserting this identity into the second equation yields the scalar problem
\[
\frac1\tau \ip{\nu(h^n)-\nu(h^{n-1})}{\xi}
+
(b+\tau)\ip{\nabla h^n}{\nabla\xi}
=
\ip{v^{n-1}}{\nabla\xi}
+
\tau\ip{f_v^n}{\nabla\xi}
+
\dual{f_h^n}{\xi}
\]
for all $\xi\in H_0^1(\Omega)$. The operator on the left-hand side is the same bounded, hemicontinuous, strictly monotone and
coercive operator as in the unforced case. The right-hand side is a
bounded linear functional on $H_0^1(\Omega)$. Hence the Browder--Minty theorem \cite{ZeidlerIIB} gives a unique
$h^n\in H_0^1(\Omega)$, and then $v^n$ is uniquely defined by the update formula.
\end{proof}

\begin{lemma}[Consistency residuals for regular solutions]\label{lem:consistency-residuals}
Let $(v,h)$ be a solution of the forced cutoff problem \eqref{eq:cutoff-westervelt-forced} such
that
\[
v\in H^1(0,T;L^2(\Omega)^d),\qquad
h\in H^1(0,T;H_0^1(\Omega)),\qquad
\nu(h)\in W^{1,1}(0,T;H^{-1}(\Omega)).
\]
Assume that the equations in \eqref{eq:cutoff-westervelt-forced} hold for a.e.\ $t\in(0,T)$ in $L^2(\Omega)^d$ and $H^{-1}(\Omega)$, respectively. We use the continuous
representatives of $v$ in $C([0,T];L^2(\Omega)^d)$ and of $h$ in
$C([0,T];H_0^1(\Omega))$\rev{.}
\begin{revision}
Because $\nu$ is globally Lipschitz, $t\mapsto\nu(h(t))$ is continuous from $[0,T]$ into
$L^2(\Omega)$. This representative agrees almost everywhere with the
$W^{1,1}(0,T;H^{-1}(\Omega))$ representative; since both are continuous as $H^{-1}(\Omega)$-valued
maps, they agree at every $t\in[0,T]$. Thus the endpoint values $\nu(h(t^n))$ below are unambiguous.
\end{revision}

For each $n=1,\dots,N$, define
\begin{equation}\label{eq:def-Rv}
R_v^n
:=
\frac1\tau\int_{t^{n-1}}^{t^n}\bigl(\nabla h(t^n)-\nabla h(s)\bigr)\,\ds
\in L^2(\Omega)^d,
\end{equation}
and let $R_\nu^n\in H^{-1}(\Omega)$ be given by
\begin{equation}\label{eq:def-Rnu}
\dual{R_\nu^n}{\xi}
:=
-\frac1\tau\int_{t^{n-1}}^{t^n}\ip{v(t^n)-v(s)}{\nabla\xi}\,\ds
+
\frac{b}{\tau}\int_{t^{n-1}}^{t^n}\ip{\nabla h(t^n)-\nabla h(s)}{\nabla\xi}\,\ds
\end{equation}
for all $\xi\in H_0^1(\Omega)$. Then the exact solution satisfies
\begin{equation}\label{eq:exact-step-v}
\frac{1}{\tau} \ip{v(t^n)-v(t^{n-1})}{\varphi}
+
\ip{\nabla h(t^n)}{\varphi}
=
\ip{f_v^n}{\varphi}
+
\ip{R_v^n}{\varphi}
\qquad \forall \varphi\in L^2(\Omega)^d,
\end{equation}
and
\begin{equation}\label{eq:exact-step-h}
\frac{1}{\tau} \dual{\nu(h(t^n))-\nu(h(t^{n-1}))}{\xi}
-
\ip{v(t^n)}{\nabla\xi}
+
b\ip{\nabla h(t^n)}{\nabla\xi}
=
\dual{f_h^n}{\xi}
+
\dual{R_\nu^n}{\xi}
\end{equation}
for all $\xi\in H_0^1(\Omega)$. Moreover,
\begin{equation}\label{eq:Rv-bound}
\tau\sum_{n=1}^N \norm{R_v^n}_{L^2(\Omega)}^2
\le
C\tau^2 \norm{h_t}_{L^2(0,T;H_0^1(\Omega))}^2,
\end{equation}
and
\begin{equation}\label{eq:Rnu-bound}
\tau\sum_{n=1}^N \norm{R_\nu^n}_{H^{-1}(\Omega)}^2
\le
C\tau^2
\left(
\norm{v_t}_{L^2(0,T;L^2(\Omega)^d)}^2
+
\norm{h_t}_{L^2(0,T;H_0^1(\Omega))}^2
\right).
\end{equation}
In addition, for every $n=1,\dots,N$,
\begin{equation}\label{eq:Rnu-max-bound}
\norm{R_\nu^n}_{H^{-1}(\Omega)}^2
\le
C\tau
\int_{t^{n-1}}^{t^n}
\Bigl(
\norm{v_t(r)}_{L^2(\Omega)^d}^2
+
\norm{h_t(r)}_{H_0^1(\Omega)}^2
\Bigr)\,dr.
\end{equation}
The constant $C$ depends on $b$ but not on $\tau$.
\end{lemma}

\begin{proof}
Integrating the first equation of \eqref{eq:cutoff-westervelt-forced} over
$(t^{n-1},t^n)$ gives, in $L^2(\Omega)^d$,
\[
\frac{v(t^n)-v(t^{n-1})}{\tau}
+
\frac1\tau\int_{t^{n-1}}^{t^n}\nabla h(s)\,\ds
=
f_v^n.
\]
Adding and subtracting $\nabla h(t^n)$ yields \eqref{eq:exact-step-v}.

Similarly, using $\nu(h)\in W^{1,1}(0,T;H^{-1}(\Omega))$ and integrating the second equation over
$(t^{n-1},t^n)$ yields, for every $\xi\in H_0^1(\Omega)$,
\[
\frac1\tau\dual{\nu(h(t^n))-\nu(h(t^{n-1}))}{\xi}
-
\frac1\tau\int_{t^{n-1}}^{t^n}\ip{v(s)}{\nabla\xi}\,\ds
+
\frac b\tau\int_{t^{n-1}}^{t^n}\ip{\nabla h(s)}{\nabla\xi}\,\ds
=
\dual{f_h^n}{\xi}.
\]
Adding and subtracting the endpoint terms $\ip{v(t^n)}{\nabla\xi}$ and
$b\ip{\nabla h(t^n)}{\nabla\xi}$ gives
\[
\begin{aligned}
&\frac1\tau\dual{\nu(h(t^n))-\nu(h(t^{n-1}))}{\xi}
-
\ip{v(t^n)}{\nabla\xi}
+
b\ip{\nabla h(t^n)}{\nabla\xi}
-
\dual{f_h^n}{\xi}
\\
&\qquad=
-\frac1\tau\int_{t^{n-1}}^{t^n}\ip{v(t^n)-v(s)}{\nabla\xi}\,\ds
+
\frac b\tau\int_{t^{n-1}}^{t^n}\ip{\nabla h(t^n)-\nabla h(s)}{\nabla\xi}\,\ds
=
\dual{R_\nu^n}{\xi}.
\end{aligned}
\]
This proves \eqref{eq:exact-step-h}.

It remains to estimate the residuals. Since $h\in H^1(0,T;H_0^1(\Omega))$, for a.e.\ $s\in(t^{n-1},t^n)$,
\[
\nabla h(t^n)-\nabla h(s)
=
\int_s^{t^n}\nabla h_t(r)\,dr
\quad\text{in }L^2(\Omega)^d.
\]
Therefore
\[
\norm{R_v^n}_{L^2(\Omega)}
\le
\frac1\tau\int_{t^{n-1}}^{t^n}
\int_s^{t^n}\norm{\nabla h_t(r)}_{L^2(\Omega)}\,dr\,ds
\le
C\tau^{1/2}
\left(
\int_{t^{n-1}}^{t^n}\norm{h_t(r)}_{H_0^1(\Omega)}^2\,dr
\right)^{1/2}.
\]
After squaring, multiplying by $\tau$, and summing over $n$, this gives
\eqref{eq:Rv-bound}.

The same argument, together with
\[
v(t^n)-v(s)=\int_s^{t^n}v_t(r)\,dr
\quad\text{in }L^2(\Omega)^d,
\]
gives, for every $\xi\in H_0^1(\Omega)$,
\[
\abs{\dual{R_\nu^n}{\xi}}
\le
C\tau^{1/2}
\left(
\int_{t^{n-1}}^{t^n}
\Bigl[
\norm{v_t(r)}_{L^2(\Omega)^d}^2
+
\norm{h_t(r)}_{H_0^1(\Omega)}^2
\Bigr]\,dr
\right)^{1/2}
\norm{\xi}_{H_0^1(\Omega)}.
\]
Taking the supremum over $\xi\in H_0^1(\Omega)$ with
$\norm{\xi}_{H_0^1(\Omega)}\le1$ proves \eqref{eq:Rnu-max-bound}. Multiplying by $\tau$ and
summing over $n$ proves \eqref{eq:Rnu-bound}.
\end{proof}

\begin{remark}[On the scope of the temporal consistency estimate]
We emphasize that the regularity assumed in Lemma~\ref{lem:consistency-residuals}, namely
$v\in H^1(0,T;L^2(\Omega)^d)$, $h\in H^1(0,T;H_0^1(\Omega))$, and
$\nu(h)\in W^{1,1}(0,T;H^{-1}(\Omega))$, is strictly stronger than the regularity provided by
Theorems~\ref{thm:existence-cutoff-H1} and~\ref{thm:existence-cutoff}. The existence of such
regular solutions is not established here and is assumed only for the purpose of the consistency
analysis.
Lemma~\ref{lem:consistency-residuals} shows that, for regular exact solutions, the residual
sequences obtained by inserting the exact solution into the Rothe step are of order $O(\tau)$ in
the discrete-in-time norms
\[
\left(\tau\sum_{n=1}^N\|\cdot\|_{L^2(\Omega)^d}^2\right)^{1/2}
\quad\text{and}\quad
\left(\tau\sum_{n=1}^N\|\cdot\|_{H^{-1}(\Omega)}^2\right)^{1/2}.
\]
These consistency bounds explain why first-order temporal behavior is plausible in stable regimes
and provide the quantitative basis for the numerical experiments below. They do not, by themselves,
constitute a convergence-rate theorem: such a theorem would additionally require a stability
estimate for the fully nonlinear coupled error system.
\end{remark}

\begin{remark}[Why no energy-norm error theorem is claimed]\label{rem:no-error-theorem}
The consistency bounds of Lemma~\ref{lem:consistency-residuals} give an $O(\tau)$ estimate for the
residual sequences in the natural discrete $L^2$-in-time norms. However, turning these bounds into
a first-order energy-norm error estimate for the coupled system requires a stability estimate which
is not a direct consequence of the monotonicity argument used above.

Indeed, let
\[
e_v^n:=v(t^n)-v^n,\qquad
e_h^n:=h(t^n)-h^n,\qquad
e_\nu^n:=\nu(h(t^n))-\nu(h^n).
\]
Testing the error equations by $e_v^n$ and $e_h^n$ gives the formal identity
\begin{align*}
&\frac{1}{2\tau}
\left(
\|e_v^n\|_{L^2(\Omega)}^2
\!-\!\|e_v^{n-1}\|_{L^2(\Omega)}^2
\!+\!\|e_v^n\!-\!e_v^{n-1}\|_{L^2(\Omega)}^2
\right)
\!+\!
\frac{1}{\tau}\ip{e_\nu^n-e_\nu^{n-1}}{e_h^n}
\!+\!
b\|\nabla e_h^n\|_{L^2(\Omega)}^2
=
\ip{R_v^n}{e_v^n}
\!+\!
\dual{R_\nu^n}{e_h^n}.
\end{align*}
The monotonicity and Lipschitz continuity of $\nu$ imply
\[
2\ip{e_\nu^n-e_\nu^{n-1}}{e_h^n}
\ge
\delta\|e_h^n\|_{L^2(\Omega)}^2
-
\frac{\Lambda^2}{\delta}\|e_h^{n-1}\|_{L^2(\Omega)}^2.
\]
After Poincaré's inequality this gives a previous-time contribution of the form
\[
-\frac{\Lambda^2 C_P^2}{\delta}\|\nabla e_h^{n-1}\|_{L^2(\Omega)}^2.
\]
Crucially, this term is not multiplied by $\tau$, whereas the parabolic dissipation available
after multiplication of the displayed error identity by $2\tau$ is
$2b\tau\|\nabla e_h^n\|_{L^2(\Omega)}^2$. Thus a direct discrete Gronwall argument in the
natural energy norm does not close uniformly as $\tau\to0$.
\end{remark}

\section{Numerical experiments}\label{sec:numerics}

We conclude by reporting numerical experiments for the cutoff Westervelt system and for the Rothe
time discretization introduced above. The goals of this section are threefold:
\begin{enumerate}[label=(\roman*)]
\item to investigate numerically the temporal convergence behavior suggested by the consistency
analysis of Section~\ref{sec:error-estimate},
\item to illustrate the a posteriori cutoff-removal mechanism in the small-data regime,
\item to compare the enthalpy-based Rothe scheme with a coefficient-frozen variant.
\end{enumerate}
This section highlights the specific advantages of the present Rothe
formulation: the monotone scalar structure of the nonlinear solve, the natural energy estimate,
and the a posteriori cutoff-inactivity mechanism.

\subsection{Fully discrete scheme}

For the spatial discretization we employ continuous piecewise affine finite elements on a
shape-regular triangulation $\mathcal T_h$ of $\Omega$, and denote by $V_h\subset H_0^1(\Omega)$
the corresponding conforming finite element space. For the velocity variable we use the
finite-dimensional space
\[
Q_h:=\{\nabla \xi_h:\xi_h\in V_h\}\subset L^2(\Omega)^d
\]
\begin{revision}
This is generally a proper gradient subspace of the elementwise constant vector fields. In one
space dimension with homogeneous endpoint values, it consists of the elementwise constants with
zero integral over $(0,1)$; it is not the full DG0 space.
\end{revision}
\rev{We denote} by $P_{Q_h}:L^2(\Omega)^d\to Q_h$ the $L^2$-orthogonal projection. The initial values
are chosen as $v_h^0=P_{Q_h}v_0$ and as a prescribed conforming approximation
$h_h^0\in V_h$ of $h_0$. At every time step, given
$(v_h^{n-1},h_h^{n-1})\in Q_h\times V_h$, we compute $h_h^n\in V_h$ from
\begin{equation}\label{eq:fem-rothe-h}
\frac{1}{\tau} \ip{\nu(h_h^n)-\nu(h_h^{n-1})}{\xi_h}
+
(b+\tau)\ip{\nabla h_h^n}{\nabla\xi_h}
=
\ip{v_h^{n-1}}{\nabla\xi_h}
+
\dual{f_h^n}{\xi_h}
+
\tau \ip{P_{Q_h}f_v^n}{\nabla\xi_h}
\end{equation}
for all $\xi_h\in V_h$, and then update
\begin{equation}\label{eq:fem-rothe-v}
v_h^n
=
v_h^{n-1}
-
\tau\nabla h_h^n
+
\tau P_{Q_h}f_v^n.
\end{equation}
Then $v_h^n\in Q_h$ for all $n$. In the unforced experiments, $P_{Q_h}f_v^n=0$.
For manufactured-solution tests with nonzero $f_v$, the projection is included in the assembly.
In the implementation, the nonlinear scalar problem \eqref{eq:fem-rothe-h} is solved by a damped
Newton iteration. \rev{The lower bound $\nu'(r)\ge\delta>0$ makes every Newton Jacobian coercive and
nonsingular, but it does not by itself imply global convergence of Newton's method. With the stated
line search and initialization, the nonlinear solves were observed to converge robustly in the
experiments below.}

All experiments are performed on the interval $\Omega=(0,1)$ with continuous piecewise affine
finite elements. The nonlinear scalar problem is solved by damped Newton iteration. The parameters used are listed in Table~\ref{tab:parameters}.

\begin{table}[htp!]
\centering
\begin{tabular}{l l}
\hline
\multicolumn{2}{l}{\emph{Common to all experiments}}\\
\hline
Spatial domain & $\Omega=(0,1)$\\
Enthalpy space $V_h$ & continuous $\mathbb P_1$ (CG1)\\
Velocity space $Q_h=\nabla V_h$ & \rev{gradient subspace of elementwise $\mathbb P_0$}\\
Nonlinear coefficient $a$ & $0.25$\\
Diffusion coefficient $b$ & $1.0$\\
Cutoff level $M$ & $1.5$ \quad($M<1/a=4$)\\
Mobility bounds & $\delta=1-aM=0.625$, \ $\Lambda=1+aM=1.375$\\
Nonlinear solver & damped Newton (SNES \texttt{newtonls}, \texttt{bt} line search)\\
Newton tolerances & $\texttt{rtol}=10^{-10}$, $\texttt{atol}=10^{-12}$, max.\ $50$ iter.\\
Linear solver & direct (LU)\\
\hline
\multicolumn{2}{l}{\emph{Exp.\ 1: temporal self-convergence (unforced)}}\\
\hline
Final time $T$ & $0.1$\\
Initial amplitude $\alpha$ & $0.05$ \quad($h_0=\alpha\sin\pi x$, $v_0=0$)\\
Fixed mesh & $n=2000$ cells\\
Time steps $\tau=T/N$ & $N\in\{16,32,64,128,256,512\}$\\
Reference step & $N_{\mathrm{ref}}=8192$\\
\hline
\multicolumn{2}{l}{\emph{Exp.\ 2: cutoff inactivity (unforced)}}\\
\hline
Final time $T$ & $0.2$\\
Mesh, time step & $n=800$, \ $\tau=10^{-3}$\\
Initial amplitudes $\alpha$ & $\{0.02,\,0.05,\,0.10,\,0.20,\,0.40\}$\\
\hline
\multicolumn{2}{l}{\emph{Exp.\ 3: enthalpy vs.\ coefficient-frozen (unforced)}}\\
\hline
Final time $T$ & $0.2$\\
Mesh, time step & $n=800$, \ $\tau=10^{-3}$\\
Initial amplitudes $\alpha$ & $\{0.40,\,0.80,\,1.00,\,1.20\}$\\
\hline
\end{tabular}
\caption{Discretization and solver parameters used in the numerical experiments.}
\label{tab:parameters}
\end{table}

When forcing terms are used, they are assembled from the prescribed exact solution by exact
differentiation and projection onto the discrete spaces. For error measurements against either an
exact or a reference solution, we report the accumulated fixed-mesh error indicator
\[
\left(
\max_{1\le m\le N}\|e_v^m\|_{L^2(\Omega)}^2
+
\tau\sum_{n=1}^N \|e_h^n\|_{L^2(\Omega)}^2
+
\tau\sum_{n=1}^N \|\nabla e_h^n\|_{L^2(\Omega)}^2
\right)^{1/2}.
\]
This quantity is motivated by the discrete energy structure of the scheme, but it should not be
interpreted as a proved energy-norm error bound for the fully nonlinear problem.

\subsection{Experiment 1: temporal self-convergence}

We investigate the temporal convergence behavior of the scheme by a self-convergence
(Cauchy-in-$\tau$) study on a fixed spatial mesh. Since Section~\ref{sec:error-estimate}
establishes temporal \emph{consistency} rather than a full energy-norm error estimate,
a self-convergence study is the natural quantitative complement. All runs use the same spatial
discretization, so the comparison primarily reflects the dependence on the time step within the
fixed finite-dimensional spatial problem. We solve the
unforced cutoff system on $\Omega=(0,1)$ with initial data $h_0=\alpha\sin(\pi x)$,
$v_0=0$, and amplitude $\alpha=0.05$, so that the cutoff is inactive throughout. On a
fixed mesh of $n=2000$ cells we compute a reference solution with $N_{\mathrm{ref}}=8192$
time steps, and for $\tau=T/N$ with $N\in\{16,32,64,128,256,512\}$ we measure, at the
common time nodes,
\[
\max_n \norm{h_h^n-h_{\mathrm{ref}}(t^n)}_{L^2(\Omega)},
\qquad
\max_n \norm{v_h^n-v_{\mathrm{ref}}(t^n)}_{L^2(\Omega)},
\]
together with the natural accumulated energy quantity defined above. Because the spatial
discretization is identical for every $\tau$, the reported rates should be understood as fixed-mesh
temporal self-convergence rates rather than as continuum error rates. The
velocity variable is represented in the discrete gradient space $Q_h=\nabla V_h$, so that
the update $v_h^n=v_h^{n-1}-\tau\nabla h_h^n$ is exact and the discrete energy identity
underlying Lemma~\ref{lem:discrete-energy} holds at the fully discrete level.

Figure~\ref{fig:time-convergence} shows the errors as a function
of $\tau$, and Table~\ref{tab:eoc} reports the experimental orders of convergence.
\begin{revision}
All three quantities display approximately first-order decay; the reported EOC values range from
$0.98$ to $1.05$. The slight values above one at the finest displayed levels may be influenced by
the finite-step reference solution, but no separate reference-resolution study is reported here.
The observed behavior is consistent with the $O(\tau)$ temporal consistency of
Section~\ref{sec:error-estimate}, without constituting a continuum error-rate theorem.
\end{revision}

\begin{figure}[htp!]
\centering
\includegraphics[width=.55\textwidth,page=1]{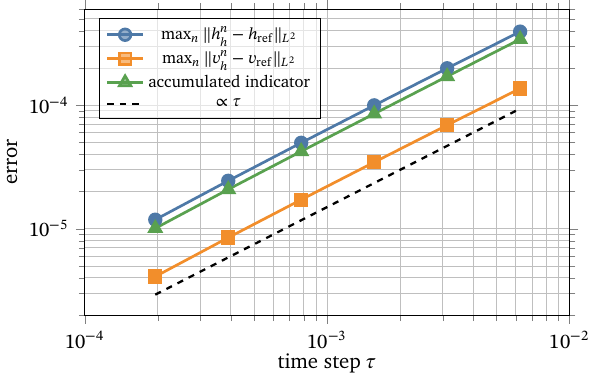}
\caption{Temporal self-convergence on a fixed mesh ($n=2000$, $N_{\mathrm{ref}}=8192$).
The enthalpy error, velocity error, and accumulated fixed-mesh indicator all decrease at first
order in $\tau$, consistent with the $O(\tau)$ temporal consistency of
Section~\ref{sec:error-estimate}. The dashed line indicates the slope-$1$ reference.}
\label{fig:time-convergence}
\end{figure}

\begin{table}[htp!]
\centering
\begin{tabular}{c c c c c c c}
\hline
$\tau$
& $\max_n \|h_h^n-h_{\mathrm{ref}}\|_{L^2}$ & EOC
& $\max_n \|v_h^n-v_{\mathrm{ref}}\|_{L^2}$ & EOC
& accumulated indicator & EOC \\
\hline
$6.250\cdot 10^{-3}$ & $3.939\cdot 10^{-4}$ & --   & $1.364\cdot 10^{-4}$ & --   & $3.419\cdot 10^{-4}$ & --   \\
$3.125\cdot 10^{-3}$ & $1.989\cdot 10^{-4}$ & 0.99 & $6.893\cdot 10^{-5}$ & 0.98 & $1.710\cdot 10^{-4}$ & 1.00 \\
$1.563\cdot 10^{-3}$ & $9.968\cdot 10^{-5}$ & 1.00 & $3.455\cdot 10^{-5}$ & 1.00 & $8.524\cdot 10^{-5}$ & 1.00 \\
$7.813\cdot 10^{-4}$ & $4.960\cdot 10^{-5}$ & 1.01 & $1.719\cdot 10^{-5}$ & 1.01 & $4.231\cdot 10^{-5}$ & 1.01 \\
$3.906\cdot 10^{-4}$ & $2.444\cdot 10^{-5}$ & 1.02 & $8.473\cdot 10^{-6}$ & 1.02 & $2.083\cdot 10^{-5}$ & 1.02 \\
$1.953\cdot 10^{-4}$ & $1.184\cdot 10^{-5}$ & 1.05 & $4.103\cdot 10^{-6}$ & 1.05 & $1.008\cdot 10^{-5}$ & 1.05 \\
\hline
\end{tabular}
\caption{Temporal self-convergence on a fixed mesh ($n=2000$, $N_{\mathrm{ref}}=8192$).
Errors are measured against a fine-step reference at common time nodes using the same spatial
discretization for every time step. The reported orders are therefore fixed-mesh temporal
self-convergence rates, not continuum error rates. All three quantities converge at first order,
consistent with the $O(\tau)$ temporal consistency of Section~\ref{sec:error-estimate}.}
\label{tab:eoc}
\end{table}

\subsection{Experiment 2: cutoff inactivity in the small-data regime}

We next illustrate the a posteriori cutoff-removal principle. We consider the unforced problem on
$\Omega=(0,1)$ with initial data
\[
h_0(x)=\alpha \sin(\pi x),
\qquad
v_0(x)=0,
\]
for several amplitudes $\alpha>0$.
For each run we monitor the quantity
\(
t^n \mapsto \norm{h_h^n}_{L^\infty(\Omega)}.
\)
In particular, if it holds
\(
\max_{0\le n\le N}\norm{h_h^n}_{L^\infty(\Omega)} < M,
\)
where $M$ denotes the cutoff threshold, then the cutoff remains inactive throughout the simulation.
\begin{revision}
At the fully discrete level, the computed equations then coincide with the corresponding uncut
first-order discretization at every time step. For the present data, $v_0=0=-\nabla\psi_0$ with
$\psi_0=0$, so the potential compatibility required for the scalar interpretation is also satisfied.
This observation is numerical evidence only; by itself it does not verify the uniform continuum
criterion of Proposition~\ref{prop:cutoff-removal-H2} or prove that every limiting weak solution is
uncut.
\end{revision}

The left panel of Figure~\ref{fig:cutoff-panels} shows the maximal discrete amplitude as a function
of the initial amplitude together with the cutoff threshold $M$. In all tested cases the numerical
solution remains strictly below that threshold. The right panel of Figure~\ref{fig:cutoff-panels}
displays the corresponding time traces. In the present one-dimensional tests, the maximal amplitude
essentially stays at the level imposed by the initial data and remains well separated from the
threshold throughout the computation. Thus the numerical results are \rev{consistent with discrete
cutoff inactivity in the tested runs, while the continuum cutoff-removal statements remain conditional
on the uniform bounds stated in Section~\ref{sec:cutoff-removal}}.

\begin{figure}[htp!]
\centering
\includegraphics[width=.4\textwidth,page=2]{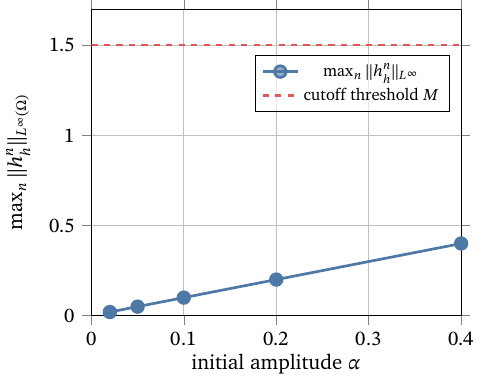}
\includegraphics[width=.4\textwidth,page=3]{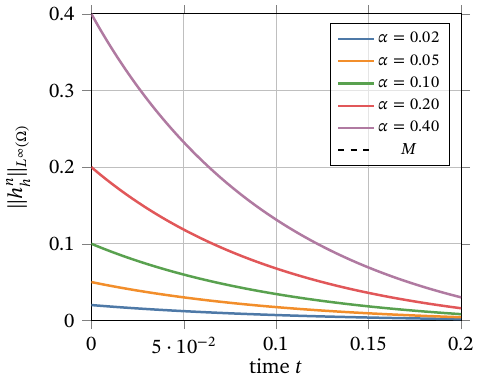}
\caption{Left: Maximum discrete amplitude as a function of the initial amplitude $\alpha$.
For all tested amplitudes, the numerical solution remains below the cutoff threshold $M$.
Right: Temporal evolution of the discrete $L^\infty$ amplitude in the cutoff-inactivity tests.
In all cases the solution remains well below the cutoff threshold throughout the computation.}
\label{fig:cutoff-panels}
\end{figure}

\subsection{Experiment 3: comparison with a coefficient-frozen scheme}

Finally, we compare the present enthalpy-based scheme with the more explicit coefficient-frozen
Rothe step
\begin{equation}\label{eq:scheme-frozen}
\frac{1}{\tau}\rev{\ip{m(h_h^{n-1})(h_h^n-h_h^{n-1})}{\xi_h}}
-
\ip{v_h^n}{\nabla\xi_h}
+
b\ip{\nabla h_h^n}{\nabla\xi_h}
=
0,
\qquad
v_h^n=v_h^{n-1}-\tau\nabla h_h^n.
\end{equation}
\begin{revision}
We consider initial data with amplitudes closer to the cutoff threshold at the single time step
$\tau=10^{-3}$ listed in Table~\ref{tab:parameters}. We compare the methods through the solver
iteration diagnostic, the maximal amplitude
$\max_n\|h_h^n\|_{L^\infty(\Omega)}$, and the common diagnostic energy
\begin{equation}\label{eq:diagnostic-energy-frozen}
\mathcal E_{\mathrm{diag}}^n
:=\int_\Omega E(h_h^n)\,dx+\frac12\|v_h^n\|_{L^2(\Omega)}^2.
\end{equation}
For the enthalpy scheme, Lemma~\ref{lem:discrete-energy} supplies a rigorous dissipation inequality.
For the coefficient-frozen scheme, \eqref{eq:diagnostic-energy-frozen} is plotted only as a common
numerical diagnostic; the remark following Definition~\ref{def:weak-solution} shows that no such
energy law holds in general. Since no time-step sweep is performed in this experiment, no conclusion
about robustness with respect to $\tau$ is drawn.
\end{revision}

\begin{revision}
Figure~\ref{fig:scheme-comparison} summarizes the comparison over several amplitudes. The
coefficient-frozen problem is linear at each time step, whereas the enthalpy scheme requires the
reported damped Newton iterations. The observed Newton counts remain modest. Because wall-clock
times and total linear-solve counts are not reported, no quantitative claim about overall computational
cost is made.

At the displayed resolution, the two methods have closely similar amplitude behavior in these tests,
and the maximal discrete amplitude stays below the cutoff threshold in every case shown.
Figure~\ref{fig:scheme-comparison-trace} gives the time-dependent comparison for $\alpha=1.2$.
The amplitude traces nearly coincide at plotting resolution, and the diagnostic energies
\eqref{eq:diagnostic-energy-frozen} have similar observed qualitative behavior. This last observation
is empirical for the fixed parameters and time step; it is not an energy-stability theorem for the
coefficient-frozen method.
\end{revision}

\begin{figure}[htp!]
\centering
\includegraphics[width=.8\textwidth,page=4]{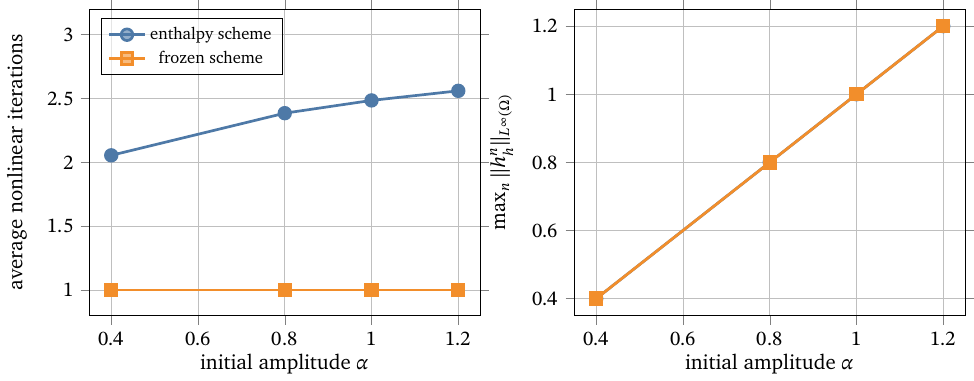}
\caption{Comparison of the enthalpy-based Rothe scheme and the coefficient-frozen scheme.
\rev{Left: solver-iteration diagnostic; the frozen problem is linear, whereas the enthalpy problem is
solved by damped Newton iteration. Right: maximal discrete amplitude
$\max_n \|h_h^n\|_{L^\infty(\Omega)}$. In the present fixed-$\tau$ test set, both schemes remain
below the cutoff threshold.}}
\label{fig:scheme-comparison}
\end{figure}

\begin{figure}[htp!]
\centering
\includegraphics[width=.8\textwidth,page=5]{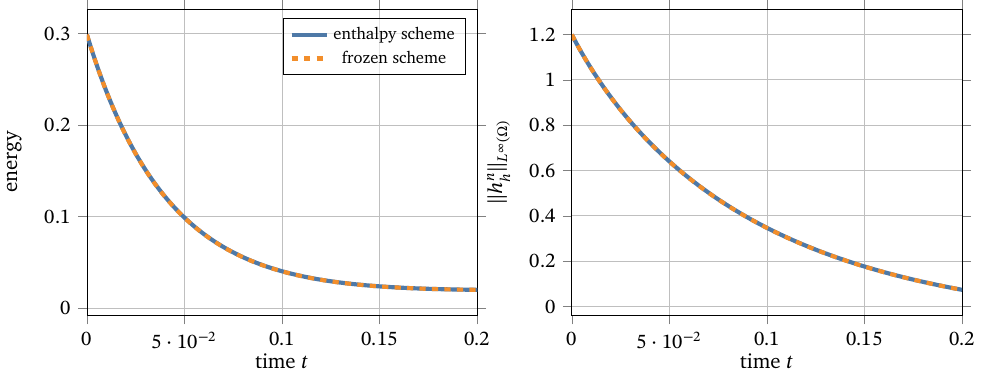}
\caption{Temporal comparison of the enthalpy-based and coefficient-frozen schemes for the
representative amplitude $\alpha=1.2$. \rev{Left: the diagnostic energy
$\mathcal E_{\mathrm{diag}}^n$ from \eqref{eq:diagnostic-energy-frozen}. Right: discrete maximum
amplitude. The two traces show similar observed qualitative behavior for this fixed test; no general
energy inequality is asserted for the frozen scheme.}}
\label{fig:scheme-comparison-trace}
\end{figure}

\bibliographystyle{ieeetr}
\bibliography{literature}

\end{document}